\documentclass[11pt,twoside]{article}
\usepackage{latexsym}
\usepackage{amssymb,amsbsy,amsmath,amsfonts,amssymb,amscd}
\usepackage{epsfig, graphicx,subfigure}
\usepackage{url}
\usepackage{subfigure}

\usepackage{pdfsync,amsthm}

\usepackage[active]{srcltx}
\usepackage{psfrag}
\usepackage{colortbl}
\newcommand{\commentout}[1]{}
\newcommand{\R}{\mathbb{R}}
\newcommand{\N}{\mathbb{N}}

\newcommand{\1}{{\mathchoice {\rm 1\mskip-4mu l} {\rm 1\mskip-4mu l}
{\rm 1\mskip-4.5mu l} {\rm 1\mskip-5mu l}}}
\newcommand {\al} {\alpha}
\newcommand {\af} {\mathfrak{a}}
\newcommand {\e}  {\varepsilon}

\newcommand {\lb} {\lambda}
\newcommand {\Lb} {\Lambda}
\newcommand {\Chi} {{\bf \raise 2pt \hbox{$\chi$}} }

\newcommand {\U} { {\mathcal U} }
\newcommand {\UU} { {\mathcal V} }
\newcommand {\f}   {\frac}
\newcommand {\p}   {\partial}

\newcommand{\dis}{\displaystyle}
\newcommand{\beq}{\begin{equation}}
\newcommand{\beqa} {\begin{array}{rl}}
\newcommand{\eeq}{\end{equation}}
\newcommand{\eeqa}{\end{array}}
\newtheorem{theorem}{Theorem}
\newtheorem{lemma}{Lemma}

\newtheorem{proposition}{Proposition}
\newtheorem{corollary}{Corollary}
\newtheorem*{theoremDG}{Theorem \cite{DG}}
\newtheorem*{theoremCL1}{Theorem \cite{CL1}}
\newtheorem*{theoremCL2}{Theorem \cite{CL2}}
\newtheorem*{corollaryMDG}{Corollary \cite{DG,M1}}

\DeclareMathOperator*{\supess}{sup\,ess}
\DeclareMathOperator*{\infess}{inf\,ess}

\newcommand{\vincent}[1]{\textbf{#1}}

\title{\LARGE Self-similarity in a General Aggregation-Fragmentation Problem ; Application to Fitness Analysis}

\author{Vincent Calvez \thanks{Ecole Normale Sup\'erieure de Lyon, UMR CNRS 5669, 46, all\'ee d'Italie, F-69364 Lyon cedex 07, France. Email: vincent.calvez@umpa.ens-lyon.fr} \hspace{1cm}
Marie Doumic Jauffret\thanks{INRIA Rocquencourt, projet BANG, Domaine de Voluceau, BP 105, F-78153 Rocquencourt, France. Email: marie.doumic-jauffret@inria.fr} \hspace{1cm}
Pierre Gabriel \thanks{Universit\'e Pierre et Marie Curie-Paris 6, UMR 7598 LJLL, BC187, 4, place de Jussieu, F-75252 Paris cedex 5, France. Email: gabriel@ann.jussieu.fr}}

\date{\today}

\begin{document}
\maketitle
\pagestyle{plain}
\pagenumbering{arabic}

\begin{abstract}
We consider the linear growth and fragmentation equation $$\dis\dfrac{\partial}{\partial t} u(x,t) + \dfrac{\partial}{
\partial x} \big(\tau(x) u\big) + \beta(x) u = 2 \int_{x}^{\infty} \beta(y) \kappa (x,y) \, u(y,t) \, dy, $$ with general coefficients $\tau,$ $\beta$ and $\kappa.$ Under suitable conditions (see \cite{DG}), the first eigenvalue represents the asymptotic growth rate of solutions, also called 
the \emph{fitness} or \emph{Malthus coefficient} in population dynamics
. This value is of crucial importance 
in understanding the long-time behavior of the population. We investigate the dependence of the dominant eigenvalue and the corresponding eigenvector on the transport and fragmentation coefficients.
We show how it behaves asymptotically 
depending on whether transport dominates fragmentation or \emph{vice versa}.
For this purpose we perform 
a suitable blow-up analysis of the eigenvalue problem in the limit of 
a small/large growth coefficient (resp. fragmentation coefficient).
We exhibit 
a possible non-monotonic dependence on the parameters, 
in contrast to what would have been conjectured on the basis of some simple cases. 
\end{abstract}

\



\

\noindent{\bf AMS Class. No.} 35B40, 35B45, 35P20, 35Q92, 45C05, 45K05, 45M05, 82D60, 92D25

\

\section*{Introduction}

The growth and division of a population of individuals  structured by a quantity 
 conserved in the division process may be described by the following growth and fragmentation equation: 
\beq\label{eq:temporel}\left \{ \begin{array}{l}
\dis\dfrac{\partial}{\partial t} u(t,x) + \dfrac{\partial}{
\partial x} \big(\tau(x) u(t,x)\big) + \beta(x) u(t,x) = 2 \int_{x}^{\infty} \beta(y) \kappa (x,y) \, u(t,y) \, dy, \qquad t \geqslant0, \quad x \geqslant0,
\\
\\
u(0,x)=u_0(x),
\\
\\
u(t,0)=0.
\end{array}\right.\eeq
This  equation is used in many different areas to model a wide range of phenomena
. The quantity $u(t,x)$ may represent a density of dust \cite{EscoMischler4}, polymers \cite{CL2,CL1}, bacteria or cells \cite{Bekkal1,Bekkal2}
. The structuring variable $x$ may be the size (\cite{BP} and references), the label \cite{Banks2,Banks}, a protein content \cite{Doumic, Magal}, a proliferating parasite content~\cite{Bansaye}; etc.
In the literature, it is referred to as the ``size-structured equation'', ``growth-fragmentation equation'', ``cell division equation'', ``fragmentation-drift equation'' or 
Sinko-Streifer model.

The growth speed $\tau=\f{dx}{dt}$ represents the natural growth of the variable $x,$ for instance by nutrient uptake or by polymerization, and the rate $\beta$ is called the fragmentation or division rate.
Notice that if $\tau$ is such that $\f1\tau$ is non integrable at $x=0,$ then the boundary condition $u(t,0)=0$ is useless.
The so-called fragmentation kernel $\kappa(x,y)$ represents the proportion of individuals of size $x\leq y$ born from a given dividing individual of size $y;$
more rigorously we should write $\kappa(dx,y),$ with $\kappa(dx,y)$ a probability measure with respect to $x$. 
For the sake of simplicity however, we 
retain the notation $\kappa(x,y) dx.$
The factor ``2'' in front of the integral term highlights the fact that we consider here binary fragmentation, namely that the fragmentation process breaks a single individual into two smaller ones. This physical interpretation leads 
us to impose the following relations 
\beq\label{as:kappa1&2}\int\kappa(x,y)dx = 1, \qquad\int x\kappa(x,y)dx  = \frac y 2,\eeq
so that $\kappa(x,y)dx$ is a probability measure and the total mass is conserved through the fragmentation process.
The method we use in this paper can be extended to more general cases where the mean number of fragments is $n_0>1$ (see \cite{DG}). 
The well-posedness of this problem as well as the existence of eigenelements has been proved in \cite{BanasiakLamb2,DG}.
Here we focus on the first eigenvalue $\lb$ associated to the eigenvector $\U$ defined by

\beq\label{eq:eigenelements}
\left \{ \begin{array}{l}
\displaystyle \f{\p}{\p x} (\tau(x) \U(x)) + ( \beta (x) + \lb) \U(x) = 2 \int_x^\infty \beta(y)\kappa(x,y) \U(y) dy, \qquad x \geqslant0,
\\
\\
\tau\U(x=0)=0 ,\qquad \U(x)>0 \; \text{ for } x>0, \qquad \int_0^\infty \U(x)dx =1.
\end{array} \right.
\eeq

The first eigenvalue $\lb$ is the asymptotic exponential growth rate of a solution to Problem \eqref{eq:temporel} (see \cite{MMP1,MMP2}).
It is often called the \emph{Malthus} parameter or 
the \emph{fitness} of the population.
Hence it is of great interest to know how it depends on the coefficients: for given parameters, is it favorable or unfavorable to increase fragmentation ?
Is it more efficient to modify the transport rate $\tau$ or to modify the fragmentation rate $\beta$ ?
Such concerns may have a deep impact on therapeutic strategy (see \cite{Bekkal1,Bekkal2,Lepoutre,Doumic}) or on the design of experimental protocols such as PMCA \footnote{\emph{PMCA}, Protein Misfolded Cyclic Amplification, is a protocol 
 designed 
 to amplify the quantity of prion protein aggregates 
 due to periodic sonication pulses. In this application, $u$ represents the density of protein aggregates and $x$ their size; the division rate $\beta$ is modulated by ultrasound waves. See Section \ref{ssec:appl:PMCA} for more details.} (see \cite{Lenuzza} and references therein).
Moreover, when modeling polymerization processes, Equation \eqref{eq:temporel} is coupled with the density of monomers $V(t),$
which appears as a multiplier for the polymerization rate (\emph{i.e.,} $\tau(x)$ is replaced by $V(t)\tau(x)$, and $V(t)$ is governed by one or more ODE - see for instance \cite{CL1,Greer,Lenuzza}).
The asymptotic study of such polymerization processes thus closely depends on such a dependence (see \cite{CL2,CL1}, where asymptotic results are obtained under the assumption of a monotonic dependence of $\lb$ with respect to the polymerization rate $\tau$).

Based on simple previously studied cases (see \cite{PG,Greer,Pruss}), 
it might be assumed intuitively that the fitness always increases when polymerization or fragmentation increases.
Nevertheless, a closer look 
reveals that 
this is not the case.

To study the dependence of the eigenproblem on its parameters, we fix coefficients $\tau$ and $\beta,$ and study how the problem is modified under the action of a multiplier of either  the growth or the fragmentation rate. We thus consider the two following problems: first,
\beq\label{eq:eigen:alpha:tau}
\left \{ \begin{array}{l}
\displaystyle \al\f{\p}{\p x} (\tau(x) \U_\al(x)) + (\beta (x) + \lb_\al) \U_\al(x) = 2 \int_x^\infty \beta(y)\kappa(x,y) \U_\al(y) dy, \qquad x \geqslant0,
\\
\\
\tau\U_\al(x=0)=0 ,\qquad \U_\al(x)>0 \; \text{ for } x>0, \qquad \int_0^\infty \U_\al(x)dx =1,
\end{array} \right.
\eeq
where $\alpha >0$ measures the strength of the polymerization (transport) term, as in the 
prion problem (see \cite{Greer}), and second
\beq\label{eq:eigen:alpha:beta}
\left \{ \begin{array}{l}
\displaystyle \f{\p}{\p x} (\tau(x) \UU_\af(x)) + ( \af \beta (x) + \Lb_\af) \UU_\af(x) = 2 \af \int_x^\infty \beta(y)\kappa(x,y) \UU_\af(y) dy, \qquad x \geqslant0,
\\
\\
\tau\UU_\af(x=0)=0 ,\qquad \UU_\af(x)>0 \; \text{ for } x>0, \qquad \int_0^\infty \UU_\af(x)dx =1,
\end{array} \right.
\eeq
where $\af >0$ modulates the fragmentation intensity, 
 as for PMCA or therapeutics applied to the cell division cycle (see the discussion in Section~\ref{sec:appl}).\\

To make things clearer, 
we give some 
insight into the dependence of $\Lambda_\af$ and $\lambda_\alpha$ on their respective multipliers $\af$ and $\alpha.$ First of all, one 
might suspect that if $\af$ vanishes or if $\al$ tends to infinity, since transport dominates, the respective eigenvectors $\U_\al$ and $\UU_\af$ tend to dilute, and on the other hand if $\af$ tends to infinity or if $\al$ vanishes, since fragmentation dominates, they tend to a Dirac mass at zero (see Figure~\ref{fig:eigenvectors} for an illustration). 
\begin{figure}[ht]
\begin{center}
\begin{minipage}{\textwidth}
\centering \includegraphics[width=.5\textwidth]
{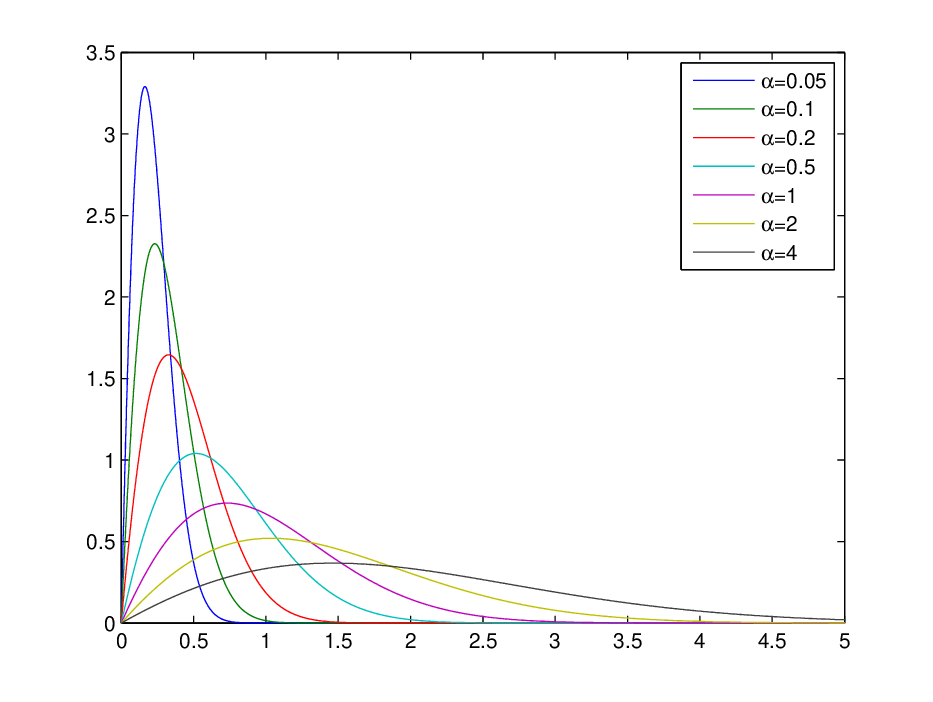}
\caption{Eigenvectors $\U_\al(x)$ for different values of $\al$ when $\beta(x)\equiv1,\ \kappa(x,y)=\f1y\1_{0\leq x\leq y}$ and $\tau(x)=x.$
In this case there is an explicit expression $\U_\al(x)=2\sqrt{\al}\left(\sqrt{\al}x+\f{\al x^2}{2}\right)\exp{\left(\sqrt{\al}x+\f{\al x^2}{2}\right)}.$ We can see that if $\alpha$ vanishes, $\U_\al$ tends to a Dirac mass, whereas it dilutes when $\alpha\to+\infty.$
}
\label{fig:eigenvectors}
\end{minipage}
\end{center}
\end{figure}
But what happens to the eigenvalues $\lb_\al$ and $\Lb_\af$ ?
Integrating Equation~\eqref{eq:eigen:alpha:beta}, we obtain the relation
$$\Lambda_\af=\af\int_0^\infty\beta(x)\UU_\af(x)\,dx$$
which gives the impression that $\Lambda_\af$ is an increasing function of $\af,$ which is true if $\beta(x)\equiv\beta$ is a constant since in this case we obtain  $\Lambda_\af=\beta\af.$
However, when $\beta$ is not a constant, the dependence of the distribution $\UU_\af(x)$ on $\af$ comes into account and we cannot conclude so easily. A better idea is given by integration of Equation~\eqref{eq:eigen:alpha:beta} against the weight $x.$ 
Together with Assumption~\eqref{as:kappa1&2}, this implies
$$\Lambda_\af\int x\UU_\af(x)\,dx=\int\tau(x)\UU_\af(x)\,dx$$
and consequently 
$$\inf_{x>0}\f{\tau(x)}{x}\leq\Lambda_\af\leq\sup_{x>0}\f{\tau(x)}{x}.$$
This last relation highlights the link between the first eigenvalue $\Lambda_\af$ and the growth rate $\tau(x),$ or more precisely $\frac{\tau(x)}{x}.$
For instance, if $\frac{\tau(x)}{x}$ is bounded, then $\Lambda_\af$ is also bounded, independently of $\af.$
Notice that in the constant case $\beta(x)\equiv\beta,$ there cannot exist a solution to the eigenvalue problem~\eqref{eq:eigenelements} for $\f{\tau(x)}{x}$ bounded since we have $\Lambda_\af=\beta\af$ which contradicts the boundedness of $\f{\tau(x)}{x}.$
In fact we check that, for $\beta$ constant, 
the existence condition~\eqref{as:betatau0} in Section~\ref{ssec:existence} imposes  that $\f{1}{\tau}$ is integrable at $x=0$ and so $\f{\tau(x)}{x}$ cannot be bounded.

Similarly, concerning Equation~\eqref{eq:eigen:alpha:tau}, an integration against the weight $x$ reads 
$$\lb_\al=\al\, \f{\int \tau \U_\al dx}{\int x\U_\al dx},$$ 
which could lead to the (false) idea that $\lb_\al$ increases with $\al,$ 
which is in fact true in the limiting case $\tau(x)=x.$ A simple integration gives more insight: this leads to
$$\lambda_\al=\int\beta(x)\U_\al(x)\,dx,\qquad \inf_{x>0}\beta(x)\leq\lambda_\al\leq\sup_{x>0}\beta(x).$$
This relation 
connects $\lambda_\alpha$ to the fragmentation rate $\beta$ when the parameter $\al$ is in front of the transport term. Moreover, 
we have seen that when the growth parameter $\al$ tends to zero, 
for instance, the distribution $\U_\al(x)$ is expected to concentrate into a Dirac mass in $x=0,$ so the identity $\lambda_\al=\int\beta(x)\U_\al(x)\,dx$  indicates that $\lambda_\al$ should tend to $\beta(0).$ Similarly, when $\al$ tends to infinity, $\lambda_\al$ should behave as $\beta(+\infty).$ 

\

These ideas on the link between $\f{\tau(x)}{x}$ and $\Lb_\af$ on the one hand, $\beta$  and $\lb_\al$ on the other hand, 
 are  expressed in a rigorous way 
below. The main assumption is  that the coefficients $\tau(x)$ and $\beta(x)$ have power-like behaviors in the neighborhood of $L=0$ or $L=+\infty,$ namely that
\beq\label{as:taunu:betagamma:L}
\exists\;\nu,\;\gamma\in\R\quad\text{such that}\quad \tau(x) \underset{x\rightarrow L}{\sim} \tau x^\nu,\quad\beta(x)\underset{x\rightarrow L}{\sim} \beta x^\gamma.
\eeq

\begin{theorem}\label{th:coself}
Under Assumption~\eqref{as:taunu:betagamma:L}, Assumptions~\eqref{as:kappa1&2}, \eqref{eq:def:kappaal}-\eqref{as:kappalim} on $\kappa,$ and Assumptions~\eqref{as:kappa3}-\eqref{as:betatauinf} stated in \cite{DG} to ensure the existence and uniqueness of solutions to the eigenproblems \eqref{eq:eigen:alpha:tau} and \eqref{eq:eigen:alpha:beta},
we have, for $L=0$ or $L=+\infty,$
$$\lim_{\al\to L}\lambda_\al=\lim_{x\to L}\beta(x)\quad \text{and}\quad \lim_{\af\to L}\Lb_\af=\lim_{x\to\f1L}\f{\tau(x)}{x}.$$
\end{theorem}

 This  is an immediate consequence of our main result, stated in Theorem~\ref{th:self} of Section~\ref{ssec:proof:self}.
As expected from the previous relations, for Problem \eqref{eq:eigen:alpha:tau} the eigenvalue behavior follows from a comparison between $\beta$ and $1$ in the neighborhood of zero if polymerization vanishes ($\al\to 0$),
and in the neighborhood of infinity if polymerization explodes ($\al\to\infty$).
For Problem \eqref{eq:eigen:alpha:beta}, it is given by a comparison between $\tau$ and $x$ (in the neighborhood of zero when $\af\to\infty$ or in the neighborhood of infinity when $\af\to0$).

It can be noticed that these behaviors are somewhat 
symmetrical:  
it is easy to see that
\begin{equation}
\Lambda_{\frac{1}{\alpha}} = \frac{\lambda_\alpha}{\alpha},\qquad \UU_{\frac{1}{\alpha}}=\U_\alpha.
\end{equation}
The first step of our proof is thus to use a properly-chosen rescaling, so that both problems \eqref{eq:eigen:alpha:tau} and \eqref{eq:eigen:alpha:beta} can be reduced to a single one, stated in Equation~\eqref{eq:dilateeP3}. 
Theorem~\ref{th:self} studies the asymptotic behavior of this new problem, 
which allows us to quantify precisely the rates of convergence of the eigenvectors toward self-similar profiles.

\

A consequence of these results is the possible non-monotonicity of the first eigenvalue as a function of $\al$ or $\af.$
In fact, if $\lim_{x\to0}\beta(x)=\lim_{x\to\infty}\beta(x)=0,$ then the function $\al\mapsto\lambda_\al$ satisfies $\lim_{\al\to0}\lambda_\al=\lim_{\al\to\infty}\lambda_\al=0$
and is positive on $(0,+\infty),$ 
because $\lambda_\al=\int\beta\U_\al>0$ for $\al>0.$
If $\lim_{x\to0}\f{\tau(x)}{x}=\lim_{x\to\infty}\f{\tau(x)}{x}=0,$ we have the same conclusion for $\af\mapsto\Lambda_\af$ (see Figure~\ref{fig:dependency} for examples).

\begin{figure}[htp]
\begin{center}
\subfigure[$\al\mapsto\lb_\al$]{\label{fig:poly}
\includegraphics[width=.48\textwidth]
{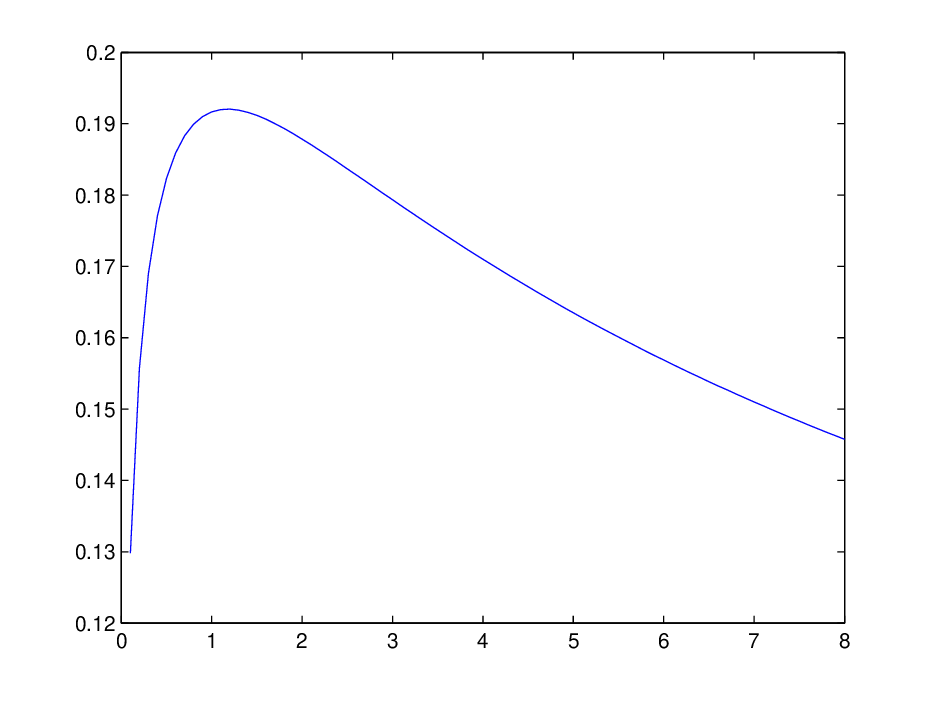}}
\subfigure[$\af\mapsto\Lb_\af$]{\label{fig:frag}
\centering \includegraphics[width=.48\textwidth]
{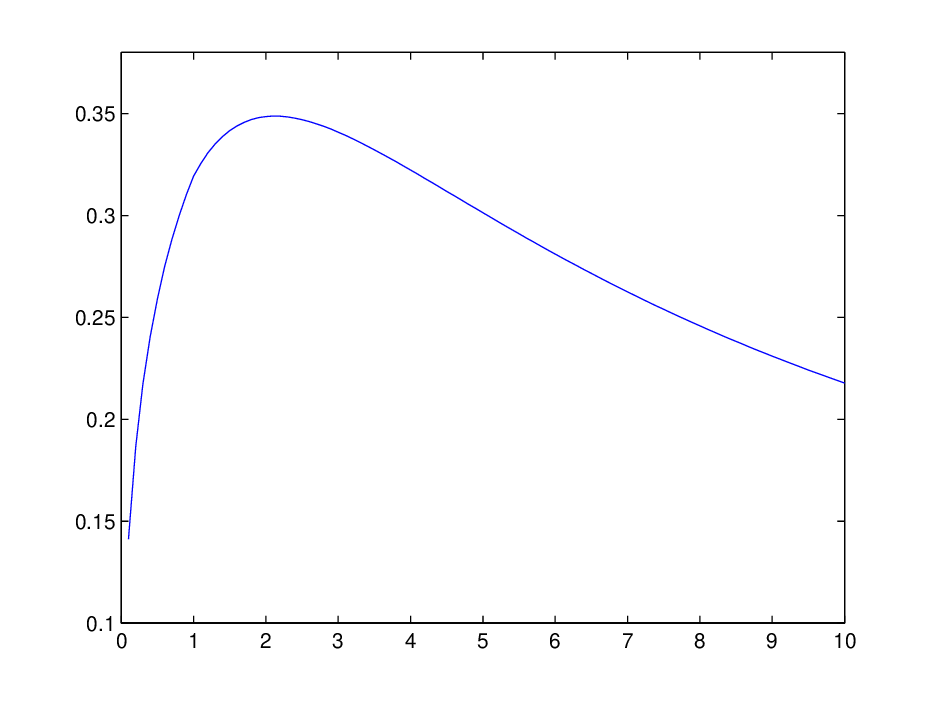}}
\end{center}
\caption{The dependences of the first eigenvalue on polymerization and fragmentation parameters for coefficients which satisfy the assumptions of Theorem~\ref{th:coself} are plotted.
The coefficients are chosen to obtain non monotonic functions.
\ref{fig:poly}: $\tau(x)=\f{8\,x^{0.2}}{1+2\,x^{4.2}},\ \beta(x)=\f{x^3}{15+x^{4.5}}$ and $\kappa(x,y)=\f1y\1_{0\leq x\leq y}.$
We have $\lim_{x\to0}\beta(x)=\lim_{x\to\infty}\beta(x)=0,$ so $\lim_{\al\to0}\lb_\al=\lim_{\al\to\infty}\lb_\al=0.$
\ref{fig:frag}: $\tau(x)=\f{1.2\,x^{1.8}}{1+2\,x^{2.8}},\ \beta(x)=\f{4\,x^2}{10+x^{0.8}}$ and $\kappa(x,y)=\f1y\1_{0\leq x\leq y}.$
We have $\lim_{x\to0}\f{\tau(x)}{x}=\lim_{x\to\infty}\f{\tau(x)}{x}=0,$ so $\lim_{\af\to0}\Lb_\af=\lim_{\af\to\infty}\Lb_\af=0.$
}
\label{fig:dependency}
\end{figure}

\

This article is organised as follows. 
In Section \ref{sec:ssresults} we state and prove the main result given in Theorem~\ref{th:self}. We first detail the self-similar change of variables that leads to the reformulation of Problems \eqref{eq:eigen:alpha:tau} and \eqref{eq:eigen:alpha:beta} in Problem \eqref{eq:dilateeP3}, as stated in Lemma \ref{lm:self}. We then recall the assumptions for the existence and uniqueness result of \cite{DG}. Here we need these assumptions  not only to have well-posed problems, 
 but also because the main tool to prove Theorem \ref{th:self} is given by estimates that are based on them. 
In Section~\ref{sec:furtherresults}, we give more precise results in the limiting cases, \emph{i.e.} when $\lim_{x\to L}\beta(x)$ or $\lim_{x\to L}\f{\tau(x)}{x}$ is finite and positive, and conversely more general results under assumptions weaker than Assumption~\eqref{as:taunu:betagamma:L}.
Finally, in Section \ref{sec:appl} we discuss how  the results might be used and interpreted in various fields of application.



\

\section{Self-similarity Result}\label{sec:ssresults}

\subsection{Self-similar transformation}\label{ssec:mainth}

The main theorem is a self-similar result, in the spirit of \cite{EscoMischler3}.
It 
considers the cases for  Equation \eqref{eq:eigen:alpha:tau} or \eqref{eq:eigen:alpha:beta}, 
and whether the parameter $\al$ or $\af$ goes to zero or to infinity.
It gathers the asymptotics of the eigenvalue and possible self-similar behaviors of the eigenvector, when $\tau$ and $\beta$ have power-like behavior in the neighborhood of $0$ or $+\infty.$ 
We first explain in detail 
how the study of both Equations \eqref{eq:eigen:alpha:tau} and \eqref{eq:eigen:alpha:beta} comes down to the study of the asymptotic behavior of a unique problem, as stated in Lemma~\ref{lm:self}. 

When fragmentation vanishes or polymerization tends to infinity, one expects the eigenvectors $\U_\al$ and $\UU_\af$ to disperse more and more.
When fragmentation tends to infinity or polymerization vanishes on the other hand, we expect them to accumulate towards zero.
This leads to the idea of performing an appropriate scaling of the eigenvector ($\U_\al$ or $\UU_\af$),
so that the rescaled problem converges toward a steady profile instead of a Dirac mass or an increasingly spread-out 
 distribution.

For given $k$ and $l,$ we define $v_\al$ or $w_\af$ by the dilation 
\beq\label{eq:defval}
v_\al(x)=\al^k\U_\al(\al^kx),\qquad w_\af(x)=\af^l\UU_\af(\af^l x).
\eeq 
The function $v_\al$ satisfies the following equation
\beq\label{eq:dilateeP}
\al^{1-k} \f{\p}{\p x} \left(\tau(\al^{k}x) v_\al(x)\right)+\left(\lb_{\al}+\beta(\al^{k}x)\right) v_\al(x)= 2\int\limits_x^\infty \beta(\al^{k}y)\kappa(\al^kx,\al^ky)v_\al(y) dy,
\eeq
and similarly the function $w_\af$ satisfies
\beq
\label{eq:dilateeF}
\af^{-l} \f{\p}{\p x} \left(\tau(\af^{l}x) w_\af(x)\right) +\left(\Lb_{\af}+\af\beta(\af^{l}x)\right) w_\af(x)=  2\af\int\limits_x^\infty \beta(\af^{l}y)\kappa(\af^l x,\af^l y)w_\af(y) dy.
\eeq
A vanishing fragmentation or an increasing polymerization will lead the mass to spread more and more, and thus lead us to consider the behavior of the coefficients $\tau,\;\beta$ around infinity. On the other hand, a vanishing polymerization or an infinite fragmentation will lead the mass to concentrate near zero, and we then consider the behavior of the coefficients around zero.
This consideration drives our main assumption~\eqref{as:taunu:betagamma:L} on the power-like behavior of the coefficients $\tau(x)$ and $\beta(x)$ in the neighborhood of $L=0$ and $L=+\infty.$
We recall this assumption here
\begin{equation*}
\exists\;\nu_L,\;\gamma_L\in\R\quad\text{such that}\quad \tau(x) \underset{x\rightarrow L}{\sim} \tau x^{\nu_L},\quad\beta(x)\underset{x\rightarrow L}{\sim} \beta x^{\gamma_L}.
\end{equation*}
In the class of coefficients satisfying Assumption~\eqref{as:taunu:betagamma:L},
Assumption~\eqref{as:betatau0} of Section~\ref{ssec:existence} is equivalent to
\beq\label{as:gammanu0}\gamma_0+1-\nu_0>0,\eeq
Assumption~\eqref{as:betatauinf} coincides with
\beq\label{as:gammanuinf}\gamma_\infty+1-\nu_\infty>0,\eeq
and in Assumption~\eqref{as:kappatau}, the condition linking $\bar\gamma$ to $\tau(x)$ becomes
\beq\label{as:gammabarnu0}\bar\gamma+1-\nu_0>0.\eeq
For the sake of simplicity, in what follows we omit the indices $_L.$ 

To preserve the fact that $\kappa$ is a probability measure, we define
\beq\label{eq:def:kappaal}\kappa_\al(x,y):=\al^k\kappa(\al^kx,\al^ky),\qquad\kappa_\af(x,y):=\af^l\kappa(\af^l x,\af^l y).\eeq
For our equations to converge, we also make the following assumption concerning the fragmentation kernels $\kappa_{\al}$ and $\kappa_{\af}:$
\beq\label{as:kappalim}\text{For }k>0,\quad \exists\; \kappa_L\quad s.t.\quad\forall\varphi\in{\mathcal C}_c^\infty(\R_+),\quad\int\varphi(x)\kappa_\al(x,y)\,dx\xrightarrow[\al\to L]{}\int\varphi(x)\kappa_L(x,y)\,dx\quad a.e.\eeq
This assumption is the convergence in a distribution sense of $\kappa_\al(.,y)$ for almost every $y.$ This is true for instance for fragmentation kernels which can be written in 
homogeneous form as $\kappa(x,y)=\f1y\tilde\kappa\bigl(\f xy\bigr).$ In this case $\kappa_\al$ is equal to $\kappa$ for all $\al,$ so $\kappa_L\equiv\kappa.$ 

Under Assumption \eqref{as:taunu:betagamma:L}, in order to obtain steady profiles, we define  
\beq\label{eq:def:taual}\tau_\al(x):=\al^{-k\nu}\tau(\al^kx),\qquad\tau_\af(x):=\af^{-l\nu}\tau(\af^l x),\qquad \beta_\al:=\al^{-k\gamma}\beta(\al^kx),\qquad \beta_\af:=\af^{-l\gamma}\beta(\af^l x).\eeq 
If $k>0,$ it provides local uniform convergences on $\R_+^*$ of $\tau_\alpha$ and $\beta_\alpha:$ $\tau_\al\underset{\al\to L}{\longrightarrow}\tau x^\nu$ and $\beta_\al\underset{\al\to L}{\longrightarrow}\beta x^\gamma.$ Equations \eqref{eq:dilateeP} and \eqref{eq:dilateeF} divided respectively by $\al^{k\gamma}$ and $\af^{l(\nu-1)}$ 
can be written as 
\beq\label{eq:dilateeP2}\al^{1+k(\nu-1-\gamma)}\f{\p}{\p x} \Bigl(\tau_\al v_\al(x)\Bigr) +\left(\al^{-k\gamma}\lb_{\al} + \beta_\al\right) v_\al(x) = 2\int\limits_x^\infty \beta_\al(y)\kappa_\al(x,y)v_\al(y) dy,\eeq
\beq\label{eq:dilateeF2}\f{\p}{\p x} \Bigl(\tau_\af w_\af(x)\Bigr) +\left(\af^{l(1-\nu)}\Lb_\af + \af^{1-l(\nu-1-\gamma)}\beta_\af\right) w_\af(x) = 2\af^{1-l(\nu-1-\gamma)}\int\limits_x^\infty \beta_\af(y)\kappa_\af(x,y)w_\af(y) dy.\eeq

In order to cancel the multipliers of $\tau_\al$ and $\beta_\af,$ it is natural to define
\beq\label{eq:def:k}
k=\f{1}{1+\gamma-\nu} >0,\qquad l=-k=\f{-1}{1+\gamma-\nu} <0,
\eeq 
which leads to
\beq\label{eq:dilateeP3}\f{\p}{\p x} (\tau_\al(x) v_\al(x)) +(\theta_\al + \beta_\al(x)) v_\al(x) = 2\int\limits_x^\infty \beta_\al(y)\kappa_\al(x,y)v_\al(y) dy,\eeq
and 
\beq\label{eq:dilateeF3}\f{\p}{\p x} (\tau_\af(x) w_\af(x)) +(\Theta_\af + \beta_\af(x)) w_\af(x) = 2\int\limits_x^\infty \beta_\af(y)\kappa_\af(x,y)w_\af(y) dy,\eeq
with
\begin{equation*}
\theta_\al=\al^{\f{-\gamma}{1+\gamma-\nu}}\lb_{\al},\qquad \Theta_\af=\af^{\f{\nu-1}{1+\gamma-\nu}}\Lb_\af.
\end{equation*}
The signs of $k$ and $l$ express the fact that for $\al>1$ or $\af<1,$ $v_\al$ and $w_\af$ are contractions of $U_{\al,\af},$ whereas for $\al<1$ or $\af>1$ they are dilations. 
It is in accordance with our initial idea of the respective roles of polymerization and fragmentation. 
Moreover, one notices that if we define $\af:=\f{1}{\al},$ then $\af^l = \al^k,$ and so Equations \eqref{eq:dilateeP3} and \eqref{eq:dilateeF3} are identical.
By the uniqueness of a solution to this eigenvalue problem, 
this implies that $\theta_\al=\Theta_{\f{1}{\al}}$ and $v_\al=w_{\f{1}{\al}}.$ We are ready to state this result in the following lemma. 
\begin{lemma}
 \label{lm:self}
Eigenproblems \eqref{eq:eigen:alpha:tau} and \eqref{eq:eigen:alpha:beta} are equivalent to the eigenproblem \eqref{eq:dilateeP3} with $k$  defined by \eqref{eq:def:k},
$\beta_\al,$ $\tau_\al$ defined by \eqref{eq:def:taual}, $\kappa_\al$ defined by \eqref{eq:def:kappaal}, $\af=\f{1}{\al}$ and the following relations linking the different problems:
\beq
\label{def:self:equiv}
v_\al (x) = \al^k \U_\al (\al^kx)=\al^k \UU_{\f{1}{\al}} (\al^kx),\qquad \theta_\al=\al^{\f{-\gamma}{1+\gamma-\nu}}\lb_{\al}=\al^{\f{1-\nu}{1+\gamma-\nu}}\Lb_{\f{1}{\al}}.
\eeq
 \end{lemma}

Defining $(v_\infty,\theta_\infty)$ as the unique solution of the following problem
 \beq\label{eq:eigen:nu:gamma}
\left \{ \begin{array}{l}
\displaystyle \f{\p}{\p x} (\tau x^\nu v_\infty(x)) + (\beta x^\gamma + \theta_\infty) v_\infty(x) = 2 \int_x^\infty \beta y^\gamma \kappa_L(x,y) v_\infty(y) dy, \qquad x \geqslant0,
\\
\\
\tau v_\infty(x=0)=0 ,\qquad v_\infty(x)>0 \; \text{ for } x>0, \qquad \int_0^\infty v_\infty(x)dx =1,\qquad \theta_\infty >0,
\end{array} \right.
\eeq
we expect $\theta_\al$ to converge towards $\theta_\infty >0$ and $v_\alpha$ towards $v_\infty$ when $\alpha$ tends to $L,$ so the expressions of $\lb_\al$, $\Lb_\af$ given by \eqref{def:self:equiv} will immediately provide their asymptotic behavior.
This result is expressed in Theorem \ref{th:self}. 
%
%
%

\subsection{Recall of existence results (\cite{DG,M1})}\label{ssec:existence}

The assumptions of the existence and uniqueness theorem for the eigenequation \eqref{eq:eigenelements} (see \cite{DG} for a complete motivation of these assumptions) 
also ensure the well-posedness of Problems \eqref{eq:eigen:alpha:tau} and \eqref{eq:eigen:alpha:beta}. 
For all $y\geq0,\ \kappa(\cdot,y)$ is a non-negative measure with a support included in $[0,y].$ We define $\kappa$ on $(\R_+)^2$ as follows: $\kappa(x,y)=0\ \text{for}\ x>y.$
We assume that for all continuous functions $\psi,$ the application $f_\psi:y\mapsto\int\psi(x)\kappa(x,y)dx$ is Lebesgue measurable.\\ 
For physical reasons, we have stated Assumption \eqref{as:kappa1&2}, so that $f_\psi\in L^\infty_{loc}(\R_+).$ Moreover we assume that the second moment of $\kappa$ is uniformly less than the first one
\beq\label{as:kappa3}\int\f{x^2}{y^2} \, \kappa(x,y) dx \leq c < 1/2.\eeq
For the polymerization and fragmentation rates $\tau$ and $\beta,$ we introduce the set
$${\mathcal P}:=\bigl\{f\geq0\,:\,\exists\mu\geq0,\ \limsup_{x\to\infty}x^{-\mu}f(x)<\infty\ \text{and}\ \liminf_{x\to\infty}x^\mu f(x)>0\bigr\}.$$
We consider
\beq\label{as:betatauspace}\beta\in L^1_{loc}(\R_+^*)\cap{\mathcal P},\qquad
\exists\, r_0\geq0\ s.t.\ \tau\in L^\infty_{loc}(\R_{+},x^{r_0}dx)\cap{\mathcal P}\eeq
satisfying
\beq\label{as:positivity}\forall K\ \text{compact of}\ (0,\infty),\ \exists\, m_K>0\quad s.t.\quad\tau(x),\beta(x)\geq m_K\ \text{for}\ a.e.\ x\in K,\eeq
\beq\label{as:kappatau}\exists\, C>0,\bar\gamma\geq0\quad s.t.\qquad\int_0^x\kappa(z,y)\,dz\leq \min\Bigl(1,C\Bigl(\f x y\Bigr)^{\bar\gamma}\Bigr)\qquad\text{and}\qquad\f{x^{\bar\gamma}}{\tau(x)}\in L^1_0.\eeq
Notice that if Assumption~\eqref{as:kappatau} is satisfied for $\bar\gamma>0,$ then Assumption~\eqref{as:kappa3} is automatically fulfilled (see Appendix A in \cite{DG}).
We assume that growth dominates fragmentation close to $L = 0$ in the following sense:
\beq\label{as:betatau0}\f{\beta}{\tau}\in L^1_0:=\bigr\{f,\ \exists a>0,\ f\in L^1(0,a)\bigl\}.\eeq
We assume that fragmentation dominates growth close to $L = +\infty$ in the following sense:
\beq\label{as:betatauinf}\lim_{x\rightarrow +\infty}\f{x\beta(x)}{\tau(x)}=+\infty.\eeq
Under these assumptions, we have the existence and uniqueness of a solution to the first eigenvalue problem 
 (see \cite{DG,M1}).

\begin{theoremDG}
Under Assumptions \eqref{as:kappa1&2}, \eqref{as:kappa3}-\eqref{as:betatauinf}, for $0<\al,\;\af <\infty,$
there exists a unique solution $(\lb,\U),$ respectively, 
to the eigenproblems \eqref{eq:eigenelements}, \eqref{eq:eigen:alpha:tau} and \eqref{eq:eigen:alpha:beta}, and we have
\begin{equation*}\begin{array}{c}
\lb>0,
\vspace{2mm}\\
x^r\tau\U\in L^p(\R_+),\quad\forall r\geq-\bar\gamma,\quad\forall p\in[1,\infty],
\vspace{2mm}\\
x^r\tau\U\in W^{1,1}(\R_+),\quad\forall r\geq0.
\end{array}\end{equation*}

\end{theoremDG}

We also recall the following corollary (first proved in \cite{M1}). We shall use it at some step of our blow-up analysis.

\begin{corollaryMDG}
Let $\tau >0,$ 
$\beta >0,$ 
$\gamma,\;\nu\in \R$ such that $1+\gamma-\nu >0,$ and $\kappa_L$ satisfy 
 Assumptions \eqref{as:kappa1&2}, \eqref{as:kappa3} and \eqref{as:kappatau}. Then there exists a unique $(\theta_\infty,v_\infty)$ solution to the eigenproblem \eqref{eq:eigenelements} with $\tau(x)=\tau x^\nu$ and $\beta(x)=\beta x^\gamma.$
\end{corollaryMDG}

In this particular case, assumptions of the above existence theorem are immediate,
and as already said both assumptions \eqref{as:betatau0} and \eqref{as:betatauinf} are satisfied if and only if $1+\gamma-\nu >0.$

\subsection{Main result}\label{ssec:proof:self}

Based on the previous sections, we can now state our main result.

\begin{theorem} \label{th:self} Let $\tau,$ $\beta$ and $\kappa$ satisfy 
 Assumptions ~\eqref{as:kappa1&2}, \eqref{as:kappa3}-\eqref{as:betatauinf}.
Let $L=0$ or $L=+\infty,$ and $\tau$ and $\beta$ satisfy 
also Assumption \eqref{as:taunu:betagamma:L}.
Let $\kappa_{\al}$ defined by \eqref{eq:def:kappaal} with $k$ defined by \eqref{eq:def:k} satisfy 
Assumption~\eqref{as:kappalim}.
Let $(v_{\al},\theta_\al)$ be the unique solution to the eigenproblem \eqref{eq:dilateeF3}.
We have the following asymptotic behaviors
$$x^rv_\al(x)\xrightarrow[\al\to L]{}x^rv_\infty(x)\ \text{strongly in}\ L^1\ \text{for all}\ r\geq0,\ \text{and}\ \theta_\al\xrightarrow[\al\to L]{}\theta_\infty.$$
\end{theorem}

Theorem~\ref{th:coself} stated in 
the introduction immediately  follows from Theorem \ref{th:self} and the expression of $\lb_\al$ and $\Lb_\af$ given by \eqref{def:self:equiv}.

\begin{proof}

It is straightforward to prove that $\kappa_\al$ satisfies Assumptions~\eqref{as:kappa1&2},~\eqref{as:kappa3} and \eqref{as:kappatau} with the same constants $c$ and $C$  as $\kappa$, thus independent of $\al.$
We have local uniform convergences in $\R_+^*$ $\tau_\al\underset{\al\to L}{\longrightarrow}\tau x^\nu$ and $\beta_\al\underset{\al\to L}{\longrightarrow}\beta x^\gamma,$ if Assumption \eqref{as:taunu:betagamma:L} holds, for $L=0$ as well as $L=\infty.$

The proof is based on uniform estimates on $v_\al$ and $\theta_\al$ independent of $\al,$ in the same spirit as in the proof of the existence theorem (see \cite{DG}). 
Once they are sufficient to bring compactness in $L^1(\R_+),$ we shall extract a converging subsequence, which will be a weak solution of Equation \eqref{eq:eigen:nu:gamma}~;  the global convergence result will then be a consequence of the uniqueness of a solution to Equation \eqref{eq:eigen:nu:gamma}.

The first step is to ensure that $\beta_\alpha,$ $\tau_\alpha$ and $\kappa_\alpha$ satisfy Assumptions  \eqref{as:betatauspace}, \eqref{as:kappatau}-\eqref{as:betatauinf} \emph{uniformly} for all $\alpha.$ In fact, they were defined for this very purpose. The reader may refer to Lemmas 
\ref{lm:tech1}, \ref{lm:tech2} and \ref{lm:tech3} in the appendix for precise statements of these uniform properties (we do not include them here since they do not present any major difficulty).

The delicate point here is to obtain fine estimates. This relies on successive and increasingly elaborate steps, which  make great use of the link between $\tau$ and $\beta$ given by Assumptions \eqref{as:kappatau}, \eqref{as:betatau0} and \eqref{as:betatauinf} to go back and forth from the transport term to the fragmentation terms.




\paragraph{\it The first estimate: $L^1$bound for $x^rv_\al$, $r\geq0$.}
For $r\geq2,$ we have by definition and 
due to Assumption~\eqref{as:kappa3}
\begin{eqnarray*}
\int_0^y \f{x^r}{y^r}\kappa_\al(x,y)\,dx&\leq&\int_0^y \f{x^2}{y^2}\kappa_\al(x,y)\,dx\\
&=&\int_0^{\al^k y} \f{x^2}{(\alpha^k y)^2}\kappa(x,\al^k y)\,dx \leq c.
\end{eqnarray*}
So, multiplying the equation  \eqref{eq:dilateeP3} on $v_\al$ by $x^r$ and then integrating on $[0,\infty),$ we find
\begin{eqnarray}\label{eq:momentbound}
\int (1-2c)x^r\beta_\al(x)v_\al(x)\,dx&\leq&r\int x^{r-1}\tau_\al(x)v_\al(x)\,dx \nonumber \\
&=&r\int_{x\leq A}x^{r-1}\tau_\al(x)v_\al(x)\,dx+r\int_{x\geq A}x^{r-1}\tau_\al(x)v_\al(x)\,dx.
\end{eqnarray}
In \eqref{eq:momentbound}, we choose $A=A_{\f{r}{\omega}}$ as defined in Lemma \ref{lm:tech1} (see the appendix) with $\omega<1-2c,$
so that for all $x\geq A$ and independently of $\al,$ we have $rx^{r-1}\tau_\al(x)v_\al(x)\leq\omega x^r\beta_\al(x)v_\al(x).$
We obtain
$$\int x^r\beta_\al(x)v_\al(x)\,dx\leq\f{r\sup_{(0,A_{\f{r}{\omega}})}\{x^{r-1}\tau_\al\}}{1-2c-\omega},$$
and the right hand side is uniformly bounded for $r-1\geq\max(r_0,-\nu)$ when $\al\to L,$ 
due to Lemma~\ref{lm:tech2} in the appendix.
Finally
\beq\label{eq:betaUL1bound}
\forall r\geq\max(2,1+r_0,1-\nu),\quad\exists C_r,\quad\int x^r\beta_\al(x)v_\al(x)\,dx\leq C_r.
\eeq
\
Moreover for all $\al$ we have $\int v_\al dx =1.$
So, once again using Lemma~\ref{lm:tech2} with $\beta^{-1}(x)=O(x^{-\gamma})$ instead of $\tau^{-1}(x)=O(x^{-\nu}),$ we conclude that uniformly in $\al\to L$
\beq\label{eq:UL1bound}x^rv_\al\in L^1(\R_+),\ \forall r\geq 0.\eeq

\paragraph{\it The second estimate: $\theta_\al$ upper bound.}
The next step is to prove the same estimate as \eqref{eq:betaUL1bound} for $0\leq r<\max(2,1+r_0,1-\nu)$ and for this we first establish a bound on $\tau_\al v_\al.$ Let $m=\max{(2,1+r_0,1-\nu)},$ then, using $\e$ and $\rho<\f{1}{2}$ defined in Lemma~\ref{lm:tech3} of the appendix and integrating \eqref{eq:dilateeP3} between $0$ and $x\leq\e,$ we find (noticing that the quantity $\tau_\al(x)v_\al(x)$ is well-defined 
because Theorem~\cite{DG} ensures that $\tau_\al v_\al$ is continuous)
\begin{eqnarray*}
\tau_\al(x)v_\al(x)&\leq&2\int_0^{x}\int\beta_\al(y)v_\al(y)\kappa_\al(z,y)\,dy\,dz\\
&\leq& 2\int\beta_\al(y) v_\al(y)\,dy\\
&=&2\int_0^\e\beta_\al(y)v_\al(y)\,dy+2\int_\e^\infty\beta_\al(y)v_\al(y)\,dy\\
&\leq&2\sup_{(0,\e)}\{\tau_\al v_\al\}\int_0^\e\f{\beta_\al(y)}{\tau_\al(y)}\,dy+ 2\e^{-m}\int_0^\infty y^{m}\beta_\al(y)v_\al(y)\,dy\\
&\leq&2\rho\sup_{(0,\e)}\{\tau_\al v_\al\}+2\e^{-m}C_{m}.
\end{eqnarray*}
Consequently, we obtain
\beq\label{eq:tauULinfbound0}\sup_{x\in(0,\e)}\tau_\al(x)v_\al(x)\leq\f{1+2C_{m}\e^{-m}}{1-2\rho}:=C.\eeq
Then we can write for any $0\leq r<m$
\begin{eqnarray*}
\int x^r\beta_\al(x)v_\al(x)\,dx&=&\int_0^\e x^r\beta_\al(x)v_\al(x)\,dx+\int_\e^\infty x^r\beta_\al(x)v_\al(x)\,dx\\
&\leq&\e^r\sup_{(0,\e)}\{\tau_\al v_\al\}\int_0^\e\f{\beta_\al(x)}{\tau_\al(x)}\,dx+\e^{r-m}\int_\e^\infty x^m\beta_\al(x)v_\al(x)\,dx\\
&\leq&C\rho\e^r+C_m\e^{r-m}:=C_r.
\end{eqnarray*}
Finally we have 
\beq\label{eq:betaUL1bound2}
\forall r\geq0,\quad\exists C_r,\quad\int x^r\beta_\al(x)v_\al(x)\,dx\leq C_r
\eeq
and so
\beq\label{eq:lambdaupperbound}\theta_\al=\int\beta_\al v_\al \leq C_0.\eeq

\paragraph{\it The third estimate: $L^\infty$bound for $x^{-\bar\gamma}\tau_\al v_\al$.}
First, integrating equation~\eqref{eq:dilateeP3} between $0$ and $x$ we find
\beq\label{eq:tauULinfbound1}\tau_\al(x)v_\al(x)\leq2\int \beta_\al(y)v_\al(y)\,dy=2\theta_\al\leq2C_0,\quad\forall x>0.\eeq
It remains to prove that $x^{-\bar\gamma}\tau_\al v_\al$ is bounded in a neighborhood of zero.\\ Let us define $f_\al:x\mapsto\sup_{(0,x)}\tau_\al v_\al.$ If we integrate \eqref{eq:eigenelements} between $0$ and $x'<x,$ we find
$$\tau_\al(x') v_\al(x')\leq2\int_0^{x'}\int\beta_\al(y) v_\al(y)\kappa_\al(z,y)\,dy\,dz\leq2\int_0^x\int\beta_\al(y) v_\al(y)\kappa_\al(z,y)\,dy\,dz$$
and so for all $x$
$$f_\al(x)\leq2\int_0^x\int\beta_\al(y) v_\al(y)\kappa_\al(z,y)\,dy\,dz.$$
Considering $\e$ and $\rho$ from Lemma~\ref{lm:tech3} in the appendix and using \eqref{as:kappatau},  for all $x<\e$ we have
\begin{eqnarray*}
f_\al(x)&\leq&2\int_0^x\int\beta_\al(y) v_\al(y)\kappa_\al(z,y)\,dy\,dz\\
&=&2\int\beta_\al(y) v_\al(y)\int_0^x\kappa_\al(z,y)\,dz\,dy\\
&\leq&2\int_0^\infty\beta_\al(y) v_\al(y)\min\Bigl(1,C\Bigl(\f x y\Bigr)^{\bar\gamma}\Bigr)\,dy\\
&=&2\int_0^x\beta_\al(y) v_\al(y)\,dy+2C\int_x^\e\beta_\al(y) v_\al(y)\Bigl(\f x y\Bigr)^{\bar\gamma}\,dy+2C\int_\e^\infty\beta_\al(y) v_\al(y)\Bigl(\f x y\Bigr)^{\bar\gamma} \,dy\\
&=&2\int_0^x\f{\beta_\al(y)}{\tau_\al(y)}\tau_\al(y) v_\al(y)\,dy+2Cx^{\bar\gamma}\int_x^\e\f{\beta_\al(y)}{\tau_\al(y)}\f{\tau_\al(y) v_\al(y)}{y^{\bar\gamma}}\,dy+2C\int_\e^\infty\beta_\al(y) v_\al(y)\Bigl(\f x y\Bigr)^{\bar\gamma}\,dy\\
&\leq&2f_\al(x)\int_0^\e\f{\beta_\al(y)}{\tau_\al(y)}\,dy+2Cx^{\bar\gamma}\int_x^\e\f{\beta_\al(y)}{\tau_\al(y)}\f{f_\al(y)}{y^{\bar\gamma}}\,dy+2C \e^{-\bar\gamma}\|\beta_\al v_\al\|_{L^1}x^ {\bar\gamma}.
\end{eqnarray*}
If we set ${\mathcal V}_\al(x)=x^{-\bar\gamma}f_\al(x),$  when $\al\to L$ we obtain
$$(1-2\rho){\mathcal V}_\al(x)\leq K_\vincent{$\e$}+2C\int_x^\e\f{\beta_\al(y)}{\tau_\al(y)}{\mathcal V}_\al(y)\,dy$$
and, 
from  Gr\"onwall's lemma, we find that $\dis{\mathcal V}_\al(x)\leq\f {K_\e e^{\f{2C\rho}{1-2\rho}}}{1-2\rho}.$
Finally we get
\beq\label{eq:0Linfbound}\sup_{(0,\e)}\{x^{-\bar\gamma}\tau_\al(x) v_\al(x)\}\leq\f {K_\e e^{\f{2C\rho}{1-2\rho}}}{1-2\rho}.\eeq
The bound~\eqref{eq:tauULinfbound1} with Assumption~\eqref{as:positivity} and the bound~\eqref{eq:0Linfbound} with Lemma~\ref{lm:tech3} 
(in which we replace $\beta_\al(x)$ by $x^{\bar\gamma}$) ensure that the family $\{v_\al\}$ is uniformly integrable.
Along with the first estimate, this result 
 ensures that $\{v_\al\}$ belongs to a compact set in $L^1$-weak 
due to the Dunford-Pettis theorem.
The sequence $\{\theta_\al\}$ also belongs  to a compact interval of $\R_+,$ so there is a subsequence of $\{(v_\al,\theta_\al)\}$ which converges in $L^1$-weak $\times\ \R.$
The limit is a solution to \eqref{eq:eigen:nu:gamma}, but such a solution is unique, 
 so the sequence converges. 
To have convergence in $L^1$-strong, we need one more estimate.

\paragraph{\it The fourth estimate: $W^{1,1}$bound for $x^r\tau_\al v_\al,\ r\geq0.$}
First, the estimates~\eqref{eq:UL1bound}~and~\eqref{eq:0Linfbound} ensure that $ x^r\tau_\al v_\al$ is uniformly bounded in $L^1$ for any $r>-1.$ Then, Equation~\eqref{eq:dilateeP3} ensures that
\beq\label{eq:W11bound}\int\bigl|\f{\p}{\p x}(x^r\tau_\al(x)v_\al(x))\bigr|\,dx\leq r\int x^{r-1}\tau_\al(x)v_\al(x)\,dx+\theta_\al+3\int\beta_\al(x)v_\al(x)\eeq
is also uniformly bounded.
For $r=0,$ the same computation works and finally $x^r\tau_\al v_\al$ is bounded in $W^{1,1}(\R_+)$ for any $r\geq0.$

\

Due to the Rellich-Kondrachov theorem, the consequence is that $\{x^r\tau_\al v_\al\}$ is compact in $L^1$-strong and so converges strongly to $\tau x^{r+\nu}v_\infty(x).$
Then, using Lemma~\ref{lm:tech2} and estimate~\eqref{eq:0Linfbound}, we can write
\begin{eqnarray*}
\int x^r|v_\al(x)-v_\infty(x)|\,dx&\leq& \int_0^\e x^r|v_\al(x)-v_\infty(x)|\,dx+\int_\e^\infty x^r|v_\al(x)-v_\infty(x)|\,dx\\
&\leq&C\int_0^\e\frac{x^{r+\bar\gamma}}{\tau_\al(x)}dx+C\int_\e^\infty x^{r+m}|\tau_\al(x)v_\al(x)-\tau x^\nu v_\infty(x)|\,dx.
\end{eqnarray*}
The first term is small for $\e$ small 
and the second term is small for $\al$ close to $L$ due to the strong $L^1$ convergence of $\{x^r\tau_\al(x)v_\al(x)\}.$
This proves the strong convergence of $\{x^rv_\al(x)\}$ and ends the proof of Theorem~\ref{th:self}.

\end{proof}

\

\section{Further Results}\label{sec:furtherresults}

In this section we examine two ways of going beyond our main result of Theorem \ref{th:self}: either by more refined assumptions than Assumption \eqref{as:taunu:betagamma:L}, and this leads to Corollary \ref{co:limit}, or by direct estimates that do not use self-similarity, and this leads to Theorem \ref{th:asympt}. A third possible direction is to closely examine Assumption \eqref{as:taunu:betagamma:L} in order to generalize Theorem \ref{th:self}; this is done in the Appendix by Proposition \ref{th:relax} (we included it to the appendix since it only slightly improves our result).

\subsection{Critical case}

When $\lim_{x\to0}\beta(x)$ or $\lim_{x\to0}\f{\tau(x)}{x}$ is a positive constant, we can enhance the result of Theorem~\ref{th:coself} if we know the higher order term in the series expansion.
Assumptions~\eqref{eq:exp:beta} and \eqref{eq:exp:tau} of Corollary~\ref{co:limit} are stronger than Assumption~\eqref{as:taunu:betagamma:L}, but provide a more precise result on the asymptotic behavior of $\lb_\al,\ \Lambda_\af.$

\begin{corollary}\label{co:limit}
If $\beta$ admits an expansion of the form
\beq\label{eq:exp:beta}\beta(x)=\beta_0+\beta_1x^{\gamma_1}+\underset{x\to0}o(x^{\gamma_1}),\quad\gamma_1>0\eeq
with $\beta_0>0$ and $\beta_1 \neq0,$ then for $\lambda$ the following expansion holds 
\beq\label{eq:exp:lb}\lambda_\al=\beta_0+\left(\beta_1\int x^{\gamma_1}v_\infty(x)\,dx\right)\,\al^{k\gamma_1}+\underset{\al\to0}o(\al^{k\gamma_1}).\eeq
In the same way, if $\tau$ admits an expansion of the form
\beq\label{eq:exp:tau}\tau(x)=\tau_0x+\tau_1x^{\nu_1}+\underset{x\to0}o(x^{\nu_1}),\quad\nu_1>1\eeq
with $\tau_0>0$ and $\tau_1\neq 0,$ then 
\beq\label{eq:exp:Lb}\Lambda_\af=\tau_0+\left(\tau_1\f{\int x^{\nu_1}v_\infty(x)\,dx}{\int xv_\infty(x)\,dx}\right)\,\af^{l(\nu_1-1)}+\underset{\af\to\infty}o(\af^{l(\nu_1-1)}).\eeq
\end{corollary}

\begin{proof}

First we assume that $\beta$ admits an expansion of the form \eqref{eq:exp:beta} and we want to prove \eqref{eq:exp:lb}.
By integrating Equation~\eqref{eq:dilateeP3} we know that $\lb_\al \int v_\alpha dx=\int\beta_\al(x)v_\al(x)\,dx$
and so multiplying by $\al^{-k\gamma_1}$ we obtain
$$\al^{-k\gamma_1}(\lb_\al-\beta_0)=\int\al^{-k\gamma_1}(\beta_\al(x)-\beta_0)v_\al(x)\,dx.$$
Now the proof is complete if we prove the convergence
\beq\label{eq:conv:beta}\int\al^{-k\gamma_1}(\beta_\al(x)-\beta_0)v_\al(x)\,dx\xrightarrow[\al\to0]{}\int\beta_1x^{\gamma_1}v_\infty(x)\,dx.\eeq
For this we use Expansion \eqref{eq:exp:beta} which provides, for all $x\geq0,$ 
\beq\label{eq:exp:beta_al}\beta_\al(x)\underset{\al\to0}{=}\beta_0+\beta_1x^{\gamma_1}\al^{k\gamma_1}+o(\al^{k\gamma_1}).\eeq
Let $m\geq\gamma_1$ such that $\dis\limsup_{x\to\infty}x^{-m}\beta(x)<\infty$ (see Assumption~\eqref{as:betatauspace}) and define
$$\dis f_\al:x\mapsto\f{\al^{-k\gamma_1}(\beta_\al(x)-\beta_0)}{x^{\gamma_1}+x^m}.$$
Due to \eqref{eq:exp:beta_al} we know 
 that  $\dis f_\al(x)\underset{\al\to0}{\longrightarrow}\f{\beta_1x^{\gamma_1}}{x^{\gamma_1}+x^m}$ for all $x.$
Moreover, due to Theorem~\ref{th:self} we  have  $(x^{\gamma_1}+x^m)v_\al(x)\xrightarrow[\al\to0]{}(x^{\gamma_1}+x^m)v_\infty(x)$ in $L^1.$
So we simply need to prove that $f_\al$ is uniformly bounded to get \eqref{eq:conv:beta} (see Section~5.2 in~\cite{Goudon}).
Due to \eqref{eq:exp:beta} and to the fact that $\dis\limsup_{x\to\infty}x^{-m}\beta(x)<\infty$ with $m\geq\gamma_1>0,$ we know that there exists a constant $C$ such that
$$|\beta(y)-\beta_0|\leq C(y^{\gamma_1}+y^m),\quad\forall y\geq0,$$
and so, because $\al\to0,$
$$\al^{-k\gamma_1}|\beta(y)-\beta_0|\leq C(\al^{-k\gamma_1}y^{\gamma_1}+\al^{-km}y^m)$$
which implies, for $x=\al^{k}y,$
$$\al^{-k\gamma_1}|\beta(\al^{k}x)-\beta_0|\leq C(x^{\gamma_1}+x^m)$$
and this proves  $f_\al(x)\leq C.$

\

The same method allows 
us to prove the result on $\Lb_\af,$ starting from the identity
$$(\Lambda_\af-\tau_0)\int xw_\af(x)\,dx=\int(\tau_\af(x)-\tau_0x)w_\af(x)\,dx$$
and using the fact that \eqref{eq:exp:tau} provides the expansion
$$\tau_\af(x)\underset{\af\to\infty}{=}\tau_0x+\tau_1x^{\nu_1}\af^{l(\nu_1-1)}+o(\af^{l(\nu_1-1)}).$$
\end{proof}

\subsection{Generalized case}

In this section, we discard Assumption~\eqref{as:taunu:betagamma:L} and give some results regarding the asymptotic behavior of the first eigenvalues for general coefficients.
The techniques used are completely different to the self-similar ones, 
but the results still provide comparisons between $\lb_\al$ and $\beta(x),$ and between $\Lb_\af$ and $\f{\tau(x)}{x}.$

\begin{theorem}
 \label{th:asympt}

\begin{enumerate}

\item {\bf Polymerization dependence.} \label{th:asympt:poly}
\begin{eqnarray}
\label{th:asympt:poly:0}
\text{If }\beta\in L^\infty_{loc}(\R_{+}^*)\text{ and }\dis\limsup_{x\to\infty}\f{\tau(x)}{x}<\infty,&\text{ then } \limsup_{\al\to0}\lambda_\al\leq\limsup_{x\to0}\beta(x).
\\ 
\label{th:asympt:poly:infty}
\text{If }\dis\f1\tau\in L_0^1:=\bigr\{f,\ \exists a>0,\ f\in L^1(0,a)\bigl\},&\text{ then }\liminf_{\al\to\infty}\lambda_\al\geq\liminf_{x\to\infty}\beta(x).
\end{eqnarray}

\item {\bf Fragmentation dependence.} \label{th:asympt:frag}
\begin{eqnarray}
\label{th:asympt:frag:0}
\text{If }\beta\in L^\infty_{loc}(\R_+),\text{ then there exists }r>0&\text{ such that }
&\limsup_{\af\to0}\Lb_\af\leq r\limsup_{x\to\infty}\f{\tau(x)}{x}.
\\
\label{th:asympt:frag:infty}
\text{If }\dis\liminf_{x\to\infty}\beta(x)>0,\text{ and if }\ \f1\tau\in L^1_0, &\text{then} &\lim_{\af\to\infty}\Lb_\af=+\infty.
\end{eqnarray}
\end{enumerate}
\end{theorem}

We first state a lemma which links the moments of the eigenvector, the eigenvalue and the polymerization rate.
\begin{lemma} \label{lm:momentr} Let $({\cal U},\lb)$ solution to the eigenproblem \eqref{eq:eigenelements}. For any $r\geq0$ we have
\beq\int x^r\U(x)\,dx\leq\f{r}{\lb}\int x^{r-1}\tau(x)\U(x)\,dx. \eeq
\end{lemma}

\begin{proof}
Integrating Equation~\eqref{eq:eigenelements} against $x^r$ we find
\begin{eqnarray*}
&-&\lefteqn{\int rx^{r-1}\tau(x)\U(x)\,dx+\lb\int x^r\U(x)\,dx+\int x^r\beta(x)\U(x)\,dx}\\
&=&2\int x^r\int_0^x\beta(y)\kappa(x,y)\U(y)\,dydx = 2\int\beta(y)\U(y)\int_0^y x^r\kappa(x,y)\,dxdy\\
&\leq &2\int\beta(y)\U(y) y^{r-1}\int_0^y x\kappa(x,y)\,dxdy = \int y^r\beta(y)\U(y)\,dy.
\end{eqnarray*}
\end{proof}

\begin{proof}[Proof of Theorem \ref{th:asympt}.\ref{th:asympt:poly}.\eqref{th:asympt:poly:0}.]
We only have to consider the case $\limsup_{x\to0}\beta(x)<\infty.$ In this case, $\beta\in L^\infty_{loc}(\R_+)$ since it is assumed  that $\beta\in L^\infty_{loc}(\R_{+}^*).$ So for $\e>0,$ we can define $\overline{\beta_\e}:=\sup_{(0,\e)}\beta(x)<\infty$ and, 
due to Assumption~\eqref{as:betatauspace}, there exist positive constants \vincent{$C_\e$} and $r$ such that $\beta(x)\leq Cx^r$ for almost every $x\geq\e.$ As a consequence,  by integration of Equation~\eqref{eq:eigen:alpha:tau} we get
$$\lambda_\al=\int\beta(x)\U_\al(x)\,dx\leq\overline{\beta_\e}+C\int x^r\U_\al(x)\,dx.$$
We can consider that $r\geq r_0+1,$ with $r_0$ defined in Assumption~\eqref{as:betatauspace}, and thus Lemma~\ref{lm:momentr} and Assumption~$\limsup_{x\to\infty}\f{\tau(x)}{x}<\infty$ lead to
$$\int x^r\U_\al(x)\,dx\leq\f{\al r}{\lambda_\al}\int x^{r-1}\tau(x)\U_\al(x)\,dx\leq\f{\al}{\lambda_\al}C\left(1+\int x^r\U_\al(x)\,dx\right)$$
for a new constant $C.$ Combining these two inequalities we obtain
$$\lambda_\al\leq\al C\left(1+\f{1}{\int x^r\U_\al(x)\,dx}\right)\leq\al C\left(1+\f{1}{\lambda_\al-\overline{\beta_\e}}\right).$$
Then, either $\lb_\al\leq\overline{\beta_\e},$ or  by multiplication by $\lb_\al-\overline{\beta_\e}>0$ we obtain 
$$\lambda_\al^2-(\overline{\beta_\e}+\al C)\lambda_\al-(1-\overline{\beta_\e})\al C\leq0,$$
and so
$$\lambda_\al\leq\f12\left(\overline{\beta_\e}+\al C+\sqrt{\overline{\beta_\e}^2+4\al C+\al^2C^2}\right).$$
Finally we have
$$\limsup_{\al\to0}\lambda_\al\leq\overline{\beta_\e}$$
and this is true for any $\e>0,$ so
$$\limsup_{\al\to0}\lambda_\al\leq\limsup_{x\to0}\beta(x).$$
\end{proof}

\begin{proof}[Proof of Theorem \ref{th:asympt}.\ref{th:asympt:poly}.\eqref{th:asympt:poly:infty}.]
Let $A>0$ and define $\underline{\beta_A}:=\inf_{(A,\infty)}\beta.$ Since $\f1\tau\in L^1_0$ and 
due to Assumption \eqref{as:positivity} we can define $I_A:=\int_0^A\f{dx}{\tau(x)}<\infty.$ Then, 
 by integration of Equation~\eqref{eq:eigen:alpha:tau} we get
\begin{eqnarray*}
\lb_\al =\int\beta(y)\U_\al(y)\,dy \geq \underline{\beta_A}\int_A^\infty\U_\al(y)\,dy =\underline{\beta_A}\Bigl(1-\int_0^A\U_\al(y)\,dy\Bigr).
\end{eqnarray*}
We know, by integration of Equation~\eqref{eq:eigen:alpha:tau} between $0$ and $x,$ that for all $x>0,\ \al\tau(x)\U_\al(x)\leq2\lb_\al.$ Thus we obtain
$$\lb_\al\geq\underline{\beta_A}\Bigl(1-\int_0^A2\lb_\al\f{dy}{\al\tau(y)}\Bigr)=\underline{\beta_A}\Bigl(1-\f{2}{\alpha} I_A  \lb_\al\Bigr),$$
and letting first $\alpha\to\infty $ and then $A\to\infty,$ as for the case \eqref{th:asympt:poly:0} above, we obtain \eqref{th:asympt:poly:infty}.
\end{proof}

\begin{proof}[Proof of Theorem \ref{th:asympt}.\ref{th:asympt:frag}.\eqref{th:asympt:frag:0}.]
The fact that $\beta\in L^\infty_{loc}(\R_+)$ 
from Assumption~\eqref{as:betatauspace} ensures the existence of two positive constants $C$ and $r$ such that for almost every $x\geq0,\ \beta(x)\leq C(1+x^r).$
So, integrating Equation~\eqref{eq:eigen:alpha:beta}, we have
$$\Lambda_\af=\af\int\beta(x)\UU_\af(x)\,dx\leq\af C\left(1+\int_{\vincent{$0$}}^\infty x^r\UU_\af(x)\,dx\right).$$
To prove \eqref{th:asympt:frag:0}, we only have to consider the case $\limsup_{x\to\infty}\f{\tau(x)}{x}<\infty.$
So, for any $A>0,$ we can define $\overline{\tau_A}:=\sup_{x>A}\f{\tau(x)}{x}<\infty.$  
Due to Lemma~\ref{lm:momentr} and considering $r\geq r_0+1$ where $r_0$ is defined in Assumption~\eqref{as:betatauspace}, we get
$$\int x^r\UU_\af(x)\,dx\leq\f{r}{\Lambda_\af}\int x^{r-1}\tau(x)\UU_\af(x)\,dx\leq\f{r}{\Lambda_\af}\left(C+\overline{\tau_A}\int x^r\UU_\af(x)\,dx\right).$$
Combining both inequalities we obtain 
$$\Lambda_\af\leq r\left(\overline{\tau_A}+\f{C}{\int x^r\UU_\af(x)\,dx}\right)\leq r\left(\overline{\tau_A}+\f{\af C^2}{\Lambda_\af-\af C}\right).$$
Then, either $\Lb_\af\leq\af C,$ or  multiplication by $\Lb_\af-\af C>0$ leads to 
$$\Lambda_\af^2-(r\overline{\tau_A}+\af C)\Lambda_\af-r(C-\overline{\tau_A})\af C\leq0,$$
and so 
$$\Lambda_\af\leq\f12\left(r\overline{\tau_A}+\af C+\sqrt{(r\overline{\tau_A})^2+4r\af C^2+\af^2C^2}\right).$$
In both cases, letting first $\af\to 0$ and then $A\to \infty,$ we obtain \eqref{th:asympt:frag:0}.
\end{proof}

\begin{proof}[Proof of Theorem \ref{th:asympt}.\eqref{th:asympt:frag:infty}.]
Let $\e>0.$ Since $\liminf_{x\to\infty}\beta(x)>0,$  
due to Assumption~\eqref{as:positivity} we have that $\underline{\beta_\e}:=\inf_{(\e,\infty)}\beta>0.$
Since $\f1\tau\in L^1_0,$  due to Assumption~\eqref{as:positivity} we get that $I_\e:=\int_0^\e\f{dx}{\tau(x)}<\infty$ and also $\lim_{\e\to0}I_\e=0.$ 
By integration of equation~\eqref{eq:eigen:alpha:beta}, we find
\begin{eqnarray*}
\Lb_\af=\af\int\beta(y)\UU_\af(y)\,dy \geq \af\underline{\beta_\e}\int_\e^\infty\UU_\af(y)\,dy
=\af\underline{\beta_\e}\Bigl(1-\int_0^\e\UU_\af(y)\,dy\Bigr).
\end{eqnarray*}
We know, as previously, 
by integration between $0$ and $x,$ that for all $x>0,\ \tau(x)\UU_\af(x)\leq2\Lb_\af.$ Thus we obtain
$$\Lb_\af\geq\af\underline{\beta_\e}\Bigl(1-\int_0^\e2\Lb_\af\f{dy}{\tau(y)}\Bigr),$$
and we get \eqref{th:asympt:frag:infty} as we previously obtained \eqref{th:asympt:poly:infty}.
\end{proof}

\

\section{Applications}\label{sec:appl}

As stated in the introduction, Problem \eqref{eq:temporel} is used to model different kinds of structured populations, so  the way its asymptotic profile depends on the parameters can be of major importance in various fields. In this section, we investigate several possible consequences of our results. In Section \ref{ssec:appl:num}, we first present the numerical scheme we used to illustrate these applications. Sections \ref{ssec:appl:prion} and \ref{ssec:appl:PMCA} focus on the Prion equation, and Section \ref{ssec:appl:cell} introduces a possible use for therapeutic optimization when Problem \eqref{eq:temporel} models the cell division cycle.

Before looking at applications, we 
recall a regularity result whose proof can be found in \cite{M2}.

\begin{lemma}\label{lm:regularity}
Under the assumptions of Section~\ref{ssec:existence}, the functions $\al\mapsto\lb_\al$ and $\af\mapsto\Lb_\af$ are well defined and differentiable on $(0,\infty).$
\end{lemma}

\subsection{Numerical scheme based on Theorem~\ref{th:self}}
\label{ssec:appl:num}
First, we  present the method we use to compute numerically the principal eigenvector $\lb$ without considering any dependence on parameters.
Then we explain how the self-similar change of variable~\eqref{eq:defval} and the convergence result of Theorem~\ref{th:self} can be used to compute the dependences $\al\mapsto\lb_\al$ and $\af\mapsto\Lb_\af,$
when parameters $\al$ and $\af$ are very large or very small.

\

The method used to compute $\lb,$ the 
solution to Equation~\eqref{eq:eigenelements}, 
 is first to compute a numerical approximation of the first eigenvector $\U,$ and then use the identity
$$\lb=\int_0^\infty \beta(x)\U(x)\,dx.$$
General Relative Entropy (GRE) introduced by~\cite{MMP1,MMP2,BP} provides the long time asymptotic behavior of any solution to the fragmentation-drift equation~\eqref{eq:temporel}.
For large times, these solutions behave like $\U(x)e^{\lambda t}$ where $\U$ and $\lambda$ are the eigenelements defined at~\eqref{eq:eigenelements}.
More precisely we have
$$\int_0^\infty |u(x,t)e^{-\lambda t}-\langle u(\cdot,t=0),\phi\rangle \U(x)|\phi(x)\,dx\xrightarrow[t\to\infty]{}0,$$
where $\phi$ is the dual eigenvector of Equation~\eqref{eq:eigenelements} (see~\cite{DG,MMP2} for more details) and $\langle u,\phi\rangle=\int_0^\infty u(x)\phi(x)\,dx.$
In \cite{CCM,LP}, it is even proved   that this convergence occurs exponentially fast under some assumptions on the coefficients.
We use this convergence to compute numerically the eigenvector $\U.$
We consider, for $u_0\in L^1(\R_+),$ 
an initial function satisfying $\int_0^\infty u_0(x)\,dx=1,$ the solution $u(x,t)$ to the fragmentation-drift equation~\eqref{eq:temporel}.
Since we do not yet know the value of $\lb,$ we define the normalized function
$$\tilde u(x,t):=\frac{u(x,t)}{\int_0^\infty u(x,t)\,dx}.$$
We can easily check that $\tilde u$ satisfies the equation
\beq\label{eq:normalized}\p_t \tilde u(x,t)+\p_x\left(\tau(x)\tilde u(x,t)\right)+\left(\int_0^\infty \beta(y)\tilde u(y,t)\,dy\right)\tilde u(x,t)+\beta(x)\tilde u(x,t)=2\int_x^\infty\beta(y)\kappa(x,y)\tilde u(y,t)\,dy,\eeq
with the boundary condition $\tau(0)\tilde u(0,t)=0,$ and that the convergence occurs
\beq\label{eq:conv_normal}\int_0^\infty |\tilde u(x,t)-\U(x)|\phi(x)\,dx\xrightarrow[t\to\infty]{}0.\eeq
The scheme used to compute $\U$ is based on the resolution of Equation~\eqref{eq:normalized} for large times and the use of \eqref{eq:conv_normal} for the stop condition.

Numerically, Equation~\eqref{eq:normalized} is solved on a truncated domain $[0,R]$ so the integration bounds have to be changed and we obtain, for $x\in[0,R],$
\beq\label{eq:truncated}\p_t \tilde u(x,t)+\p_x\left(\tau(x)\tilde u(x,t)\right)+\left(\int_0^R \beta(y)\tilde u(y,t)\,dy\right)\tilde u(x,t)+\beta(x)\tilde u(x,t)=2\int_x^R\beta(y)\kappa(x,y)\tilde u(y,t)\,dy.\eeq
What we lose when we solve this truncated equation are the integral terms $\int_R^\infty\beta(x)\tilde u(x,t)\,dx$ and $\int_R^\infty\beta(y)\kappa(x,y)\tilde u(y,t)\,dy,$
and the outgoing flux $\tau(R)\tilde u(R,t)$ at the boundary $x=R.$
To be as close as possible to the non-truncated solution, we need to choose a sufficiently large $R$  so that these quantities are small enough.
It is proved in \cite{DG} that $\beta(x)\U(x)$ and $\tau(x)\U(x)$ are fast decreasing when $x\to+\infty.$
When $\tilde u$ is close to the equilibrium $\U,$ the value of $R$ has to be adapted so that  $\tau(x)\tilde u(x,t)$ and $\beta(x)\tilde u(x,t)$ be smaller than a fixed parameter $\epsilon$ for $x$ close to $R.$
Parameter $\epsilon$ is expected to be very small, and it is also used for the stop condition~\eqref{eq:stop}.

We assume that $[0,R]$ is divided into $N$ uniform cells and we denote $x_i=i\Delta x$ for $0\leq i\leq N$ with $\Delta x=\f{R}{N}.$
The time is discretized with the time step $\Delta t$ and we denote $t^n=n\Delta t$ for $n\in\N.$
We adopt the finite difference point of view, namely we compute an approximation $\tilde u_i^n$ of $\tilde u(x_i,t^n).$
It remains to explain how we go from the time $t^n$ to the time $t^{n+1}.$
To enforce that $\sum_{i=1}^N \tilde u_i^n=1$ at each time step, we split the evolution into two steps.
First we compute, from $(\tilde u_i^n)_{1\leq i\leq N},$ a vector $(u_i^{n+1})_{1\leq i\leq N}$ which is obtained with the formula
\begin{equation}\label{eq:numsch}
\f{u_i^{n+1}-\tilde u_i^n}{\Delta t}=-\f{\tau_i\tilde u_i^{n+1}-\tau_{i-1}\tilde u_{i-1}^n}{\Delta x}-\beta_iu_i^{n+1}+2\Delta x\sum_{j=1}^{N}\beta_j\kappa_{i,j}\tilde u_j^{n},
\end{equation}
where $\beta_i=\beta(x_i),\ \tau_i=\tau(x_i)$ and $\kappa_{i,j}=\kappa(x_i,x_j).$
This is a semi-implicit Euler discretization of the growth-fragmentation equation~\eqref{eq:temporel}.
We choose this scheme to ensure  stability without any CFL condition, since the scheme is positive.
Then we set
$$\tilde u_i^{n+1}:=\frac{u_i^{n+1}}{\Delta x\sum_{j=1}^{N} u_j^{n+1}}$$
and the discrete integral of $(\tilde u_i^{n+1})_{1\leq i\leq N}$ satisfies $\Delta x\sum_{i=1}^N \tilde u_i^{n+1}=\Delta x\sum_{i=1}^N \tilde u_i^{n}=1.$
Using the $L^1$ convergence \eqref{eq:conv_normal}, we end the algorithm when
\beq\label{eq:stop}\f{\Delta x}{\Delta t}\sum_{i=1}^{N}|\tilde u_i^{n}-\tilde u_i^{n-1}|<\epsilon\eeq
where $\epsilon\ll\Delta x.$
Then 
$$\lb\simeq\Delta x\sum_{i=1}^N \beta_i\tilde u_i^n.$$

The semi-implicit scheme~\eqref{eq:numsch} is efficient to avoid oscillations on the numerical solution, 
 but it is not conservative. 
 This scheme has to be avoided if we want to solve Equation~\eqref{eq:temporel} for any time.
Here we are only interested in the steady state of Equation~\eqref{eq:normalized}, so 
non-conservation does not matter, 
 because the steady state is the same for an implicit or an explicit scheme.

\

Now we want to compute $\lambda_\alpha$ and $\Lambda_\af$ for a large range of $\alpha$ and $\af.$
According to the discussion in 
the introduction, the eigenvectors $\U_{\al}$ and $\UU_\af$ are concentrated at the origin for $\al$ small or $\af$ large and, conversely,
spread out for $\al$ large or $\af$ small.
Then, to avoid an adaptation of the truncation 
parameter $R$ or an adaptation of the discretization size step $\Delta x$ when $\alpha$ and $\af$ vary,
we compute $\theta_\al$ defined in Equation~\eqref{eq:dilateeP3}.
To this end, we need to compute the dilated eigenvector $v_\al$ defined in \eqref{eq:defval} which converges to a fixed profile $v_\infty$ when $\al\to L,$ as stated in Theorem~\ref{th:self}.
This convergence ensures that the vector $v_\al$ does neither disperses nor concentrates too much when $\al$ varies,
and so we can find a truncation 
and a size step which 
work for any $\al\to L.$
It remains to distinguish $L=0$ from $L=+\infty$ by dividing $(0,+\infty)$ into two sets: for instance $(0,1]$ and $(1,+\infty).$
For $0<\alpha<1$, we use the dilation coefficient $k$ associated to $\nu$ and $\gamma$ such that $\tau(x)\underset{x\to0}{\sim}x^\nu$ and $\beta(x)\underset{x\to0}{\sim}x^\gamma.$
For $\alpha>1$ we do the dilation associated to $\nu$ and $\gamma$ such that $\tau(x)\underset{x\to\infty}{\sim}x^\nu$ and $\beta(x)\underset{x\to\infty}{\sim}x^\gamma.$
Finally, we use Equation~\eqref{def:self:equiv} in Lemma~\ref{lm:self} to recover $\lb_\al$ or $\Lb_\af$ from the numerical value of $\theta_\al$ (see Figure \ref{fig:dependency} for a numerical illustration).

\

All the figures in the paper were obtained using this numerical scheme.

\subsection{Steady States of the Prion Equation}\label{ssec:appl:prion}
To model polymerization processes, Equation~\eqref{eq:temporel} can be coupled to an ODE which 
incorporates the evolution of the quantity of monomers.
The so-called ``prion equation'' (see~\cite{Greer,LW,Pruss}) is
\begin{equation}
\label{eq:Greer}
\left\{
\begin{array}{rl}
\dfrac{dV(t)}{dt}&=\displaystyle \xi- V(t)\left[\delta+ \int_{0}^{\infty} \tau(x) u(x,t) \; dx\right],
\vspace{.2cm}\\
\dfrac{\partial}{\partial t} u(x,t) &= \displaystyle - V(t) \frac{\partial}{\partial x} \big(\tau(x) u(x,t)\big) - [\beta(x)+\mu(x)] u(x,t) + 2 \int_{x}^{\infty} \beta(y) \kappa (x,y) \, u(y,t) \, dy,
\vspace{.2cm}\\
u(0,t) &=0,
\end{array} \right.
\end{equation}
where the quantity of monomers is denoted by $V(t).$
In this model, the monomers are prion proteins, produced and degraded by the cells with rates $\xi$ and $\delta,$ and attached to polymers of size $x$ with respect to the rate $\tau(x).$
The polymers are fibrils of misfolded pathogenic proteins, which have the ability to transconform normal proteins (monomers) into abnormal ones by a polymerization process, 
which is not yet very well understood.
The size distribution of polymers $u(x,t)$ is the solution to the growth-fragmentation~equation~\eqref{eq:temporel} in which $V(t)$ is added as a multiplier for the polymerization rate.
A degradation rate $\mu(x)$ is also considered for the polymers. For the sake of simplicity, this rate is assumed to be size-independent in the following study $(\mu(x)\equiv\mu_0).$

Equation~\eqref{eq:Greer} models the proliferation of prion disease.
An individual is said to be infected by prion disease when polymers of misfolded proteins are present, namely when $u(\cdot,t)\not\equiv0$ at the 
time $t.$

The coupling between $V(t)$ and $u(t,x)$ appears in the equation for $u$ as a modulation of the polymerization rate. One can immediately see  the link with the eigenproblem \eqref{eq:eigen:alpha:tau} satisfied by $\U_\al:$  $\U_\al$ is the principal eigenvector  linked to the linearization of the prion equation around a fixed monomer quantity $ V=\alpha.$ 
Investigating the dependence of the fitness $\lb_V$ 
with respect to the  polymerization and fragmentation coefficients 
is a first step towards a better understanding of the  propagation of the disease.
It has been reported that the course of prion infection in the brain follows heterogeneous patterns.
It has been postulated that the neuropathology of prion infection could be related to different kinetics in different compartments of the brain \cite{Collinge}.

Modeling the propagation of prion in the brain requires a good understanding of possible dynamics (\emph {e.g.} monostable, bistable etc).
Such a study can be done through the dependence of the first eigenvalue on parameters~\cite{PG2}.
In \cite{CL1,Greer}, it is shown that, under some conditions, the coexistence of two steady states can occur (one endemic and one disease-free).

A steady state $(V_\infty,u_\infty(x))$ is a solution to
\begin{equation}
\label{eq:steady}
\left\{
\begin{array}{rl}
0&=\displaystyle \xi- V_\infty\left[\delta+ \int_{0}^{\infty} \tau(x) u_\infty(x) \; dx\right],
\vspace{.2cm}\\
\mu_0u_\infty(x) &= \displaystyle - V_\infty \frac{\partial}{\partial x} \big(\tau(x) u_\infty(x)\big) - \beta(x) u_\infty(x) + 2 \int_{x}^{\infty} \beta(y) \kappa (x,y) \, u_\infty(y) \, dy,
\vspace{.2cm}\\
u_\infty(0) &=0.
\end{array} \right.
\end{equation}
The disease-free steady state corresponds to the solution without any polymer $\left(\overline V=\f{\xi}{\delta},\overline u\equiv0\right).$
Other steady states can exist and are called endemic or disease steady states.
They are solutions to System~\eqref{eq:steady} with $V_\infty>0$ and $u_\infty\not\equiv0$ nonnegative.
To know if such disease steady states exist, we recall briefly here the method of \cite{CL2,CL1}.
A positive steady state $u_\infty$ can be seen as an eigenvector solution of \eqref{eq:eigen:alpha:tau} with $\al=V_\infty$ such that \beq\label{eq:lambda=mu}\lb_\al=\lambda_{V_\infty}=\mu_0.\eeq
This shows the crucial importance of a study of the map $V\mapsto\lambda_V.$
Any value $V_\infty$ solution to \eqref{eq:lambda=mu} provides a size distribution of polymers
$$u_\infty(x)=\varrho_\infty\U_{V_\infty}(x).$$
The quantity of polymers $\varrho_\infty$ is then prescribed by the equation 
for monomers and has to satisfy the relation
$$\varrho_\infty=\f{\xi V_\infty^{-1}-\delta}{V_\infty\int\tau(x)\U_{V_\infty}(x)\,dx}.$$
The quantity of polymers has to be positive, 
which is equivalent to the condition
$$V_\infty<\f{\xi}{\delta}=\overline V.$$
Finally, the disease steady states correspond exactly to the zeros of the map $V\mapsto\lb_V-\mu_0$ in the interval $(0,\overline V).$
Due to the different results of Sections~\ref{sec:ssresults}~and~\ref{sec:furtherresults}, we know that this map is not necessarily monotonic, as is assumed in \cite{CL1}.
Thus, by continuity of the dependence of $\lb$ on $V$ (see Lemma~\ref{lm:regularity}), there can exist several disease steady states for a well-chosen $\mu_0$ and a large enough $\overline V=\f\xi\delta.$ 
This point is illustrated in the example below, where there exist two disease steady states.

We can investigate the stability of the disease-free steady state through the results obtained in~\cite{CL2,CL1}.
For this, we introduce the dual eigenvector $\overline\varphi$ of the growth-fragmentation operator with the transport term $\overline V$
\beq\label{eq:dual}
\left \{ \begin{array}{l}
\displaystyle -\overline V\tau(x)\f{\p}{\p x} ( \overline\varphi(x)) + ( \beta (x) + \lb) \overline\varphi(x) = 2 \beta(x)\int_0^x\kappa(y,x) \overline\varphi(y) dy, \qquad x \geqslant0,
\\
\\
\overline\varphi(x)\geq0,\qquad \int_0^\infty \overline\varphi(x)\U_{\overline V}(x)dx =1.
\end{array} \right.\eeq
We assume that we have a case when there exist two constants $K_1$ and $K_2$ such that
\beq\label{as:phi}\left|\tau(x)\f{\p}{\p x}\overline\varphi(x)\right|\leq K_1\overline\varphi(x),\quad\text{and}\quad\tau(x)\leq K_2\overline\varphi(x).\eeq
This assumption generally holds true when $\f{\tau(x)}{x}$ is bounded because $\overline\varphi$ grows linearly at infinity according to general properties proved in~\cite{DG,M1,BP,PR}.
Then we can reformulate the theorems of~\cite{CL2,CL1}.

\begin{theoremCL1}[Local stability]
Suppose that assumption~\eqref{as:phi} holds true and that $\lb_{\overline V}<\mu_0.$ Then the steady state $(\overline V,0)$ is locally non-linearly stable.
\end{theoremCL1}

\begin{theoremCL2}[Persistence]
Suppose that assumption~\eqref{as:phi} holds, 
$V(0)\leq\overline V,\ \int_0^\infty(1+x)u(t,x)\,dx$ is uniformly bounded, and that $\lb_{\overline V}>\mu_0.$
Then the system remains away from the steady state $(\overline V,0).$ More precisely we have:
$$\liminf_{t\to\infty}\int_0^\infty\overline\varphi(x)u(x,t)\,dx>0.$$
\end{theoremCL2}

\

{\it\underline{Example.}}
Let us consider the same coefficients as in Figure~\ref{fig:poly}.
We can choose a small enough $\mu_0$  to ensure the existence of two values $V_1<V_2$ such that $\lb_{V_1}=\lb_{V_2}=\mu_0.$
As a consequence, we know 
due to the previous study that there exists no disease steady state if $\overline V<V_1,$ one if $V_1<\overline V<V_2,$ and two if $\overline V>V_2.$
Concerning the stability of the disease-free steady state, we first notice that the fragmentation rate $\beta(x)$ satisfies the assumption $\beta(x)\leq A+Bx,$ 
which is sufficient to ensure that $\int_0^\infty(1+x)u(t,x)\,dx$ is uniformly bounded (see \cite{CL2} Theorem 2.1).
Thus we can apply the previous theorems so that $(\overline V,0)$ is stable if $\overline V<V_1,$ unstable if $V_1<\overline V<V_2,$ and recovers its (local) stability if $\overline V>V_2.$
In Figure~\ref{fig:steadystates}, the graph of the negative fitness $V\mapsto\mu_0-\lb_V$ is plotted
(because the quantity of polymers influences the evolution of $V(t)$ with a negative contribution) and the zones of stability and unstability for $\overline V$ are pointed out.
The non intuitive conclusion is that an increase in the production rate $\xi$ or a decrease in the death rate $\delta$ can stabilize the disease-free steady state.
In this situation, what happens  is that the largest polymers are the most stable since $\lim_{x\to\infty}\beta(x)=0$ (this situation is biologically relevant, see for instance \cite{Silveira}).
When the number of polymers is large, the polymerization is strong and it 
results in long stable polymers.
Because they do not break easily, their number does not increase very fast, \emph{i.e.} the fitness of the polymerization-fragmentation equation is small.
But the degradation term is assumed to be size-independent, and then the fitness $\lb_V$ becomes smaller than $\mu_0$ for a sufficiently large $V$.
This phenomenon stabilizes the disease-free steady state because, when polymers are injected in a cell, they tend to disappear immediately, 
 since $\lb_{\overline V}<\mu_0.$

\begin{figure}[ht]
\begin{center}
\begin{minipage}{\textwidth}
\psfrag{lambda}[l]{$\mu_0-\lambda_V$}
\psfrag{Vbar}[l]{$\bar V$}
\psfrag{V1}[l]{$V_1$}
\psfrag{V2}[l]{$V_2$}
\psfrag{V}[l]{$V$}
\centering \includegraphics[width=.8\textwidth]
{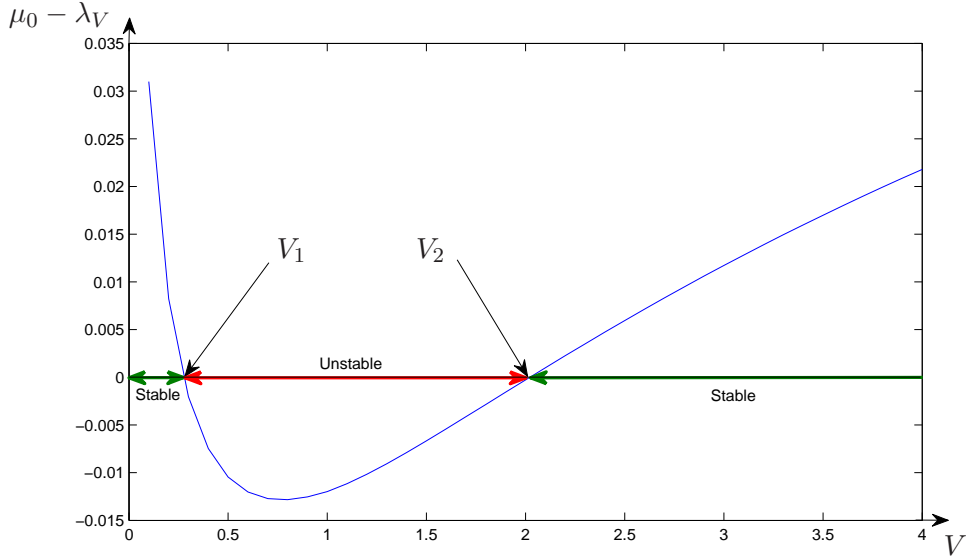}
\caption{The negative fitness $V\mapsto\mu_0-\lb_V$ is plotted for the same coefficients as in Figure~\ref{fig:poly}.
The zeros $V_1$ and $V_2$ correspond to disease steady states and separate the areas of stability or unstability of the disease-free steady state.
}
\label{fig:steadystates}
\end{minipage}
\end{center}
\end{figure}

As regards the stability of the disease steady states, the study is much more complicated.
Nevertheless, we can imagine that $V_1$ is stable and $V_2$ unstable.
This postulate is based on Figure~\ref{fig:steadystates}, on stability results for similar problems \cite{PG2} and on the results obtained in~\cite{Engler, Pruss} in a case where System~\eqref{eq:Greer} can be reduced to a system of ODEs.
For the coefficients considered in~\cite{Engler,Pruss}, $V\mapsto\lambda_V$ is an increasing squareroot function so there is only one value $V_\infty$ such that $\lambda_{V_\infty}=\mu_0.$
Moreover, the dynamics of the solutions is entirely determined: either $\overline V\leq V_\infty$ and the disease-free steady state is the only steady state and is globally asymptotically stable,
or $\overline V>V_\infty$ and the endemic steady state is globally asymptotically stable.
This very strong result means that, if $\overline V\leq V_\infty,$
the individual is resistent to the disease because the misfolded prion proteins ultimately disappear even if a very large quantity is injected.
On the other hand, if $\overline V>V_\infty,$
the individual is very sensitive to the prion disease because the system converges to the endemic steady state as soon as a minute quantity of misfolded proteins is injected.
This alternative is no longer true in the case of Figure~\ref{fig:steadystates} with $\overline V>V_2:$
 in this case, the disease-free steady state is locally stable but coexists with two endemic steady states (with the one corresponding to $V_1$ which is probably locally stable).
So the individual can resist to an injection of abnormal prion proteins if the quantity is small enough (local stability of $\overline V),$
but the injection of a large number of polymers can make the system switch to the endemic equilibrium associated to $V_1.$
Such a  bistability situation has already been exhibited for other models of prion proliferation (see~\cite{Laurent}) but never for the polymerization model~\eqref{eq:Greer}.

\subsection{Optimization of the PMCA protocol
}\label{ssec:appl:PMCA}

Prion diseases, briefly described in Section~\ref{ssec:appl:prion} (see \cite{Lenuzza} for more details), are fatal, infectious and neurodegenerative diseases with a long incubation period.
They include bovine spongiform encephalopathies in cattle, scrapie in sheep and Creutzfeldt-Jakob disease in humans.
It is therefore of importance to be able to diagnose infected individuals to avoid the spread of the disease in a population.
But the dynamics of proliferation is slow and the amount of prion proteins is low at the beginning of the disease.
Moreover, these proteins are concentrated in vital organs like the brain, and are present in only minute quantities in tissues like the blood.
To be able to detect prions in these tissues, a solution is to amplify their quantity.
A promising recent technique of amplification is PMCA (Protein Misfolded Cyclic Amplification).
Nevertheless, this protocol 
is not able to amplify prions for all the prion diseases from tissues with low infectivity.
PMCA diagnosis remains to be improved and mathematical modeling and analysis can help to do so.

\

PMCA is an \emph{in vitro} cyclic process that quickly amplifies very small quantities of prion proteins present in a sample.
In this sample, the pathogenic proteins (polymers) are put in the presence of a large quantity of normal proteins (monomers).
Then the protocol 
consists in 
alternating two phases:

- a phase of incubation during which the sample is left to rest and the polymers can attach the monomers (increasing the size of the polymers),

- a phase of sonication during which waves are sent on the sample in order to break the polymers into numerous smaller ones (increasing the number of the polymers).

To model this process, we can use the growth-fragmentation equation~\eqref{eq:temporel} as in~\eqref{eq:Greer}.
The main difference is that the PMCA takes place \emph{in vitro}, 
and there is no production of monomers.
As there is such a large number of monomers, in order to improve the polymerization, we can neglect their consumption by the polymerization process and assume that their concentration remains constant during the PMCA.
We now introduce the ``sonication'' into the equation.
Because the sonication phase increases the fragmentation of polymers, an initial modeling can be to add a time-dependent parameter $\af(t)$ in front of the fragmentation parameter $\beta(x).$
Then the 
alternating incubation-sonication phases correspond to a rectangular function $\af(t)$ which is equal to $1$ during the incubation time (since the sample is left to rest),
and $\af_{max}$ during the sonication pulse (where $\af_{max}$ represents the maximal power of the sonicator).
We obtain the model
\beq\label{eq:PMCA}
\dfrac{\partial}{\partial t} u(x,t) = \displaystyle - V_0\frac{\partial}{\partial x} \big(\tau(x) u(x,t)\big) - \af(t)\beta(x) u(x,t) + 2 \af(t)\int_{x}^{\infty} \beta(y) \kappa (x,y) \, u(y,t) \, dy,
\eeq
where $u(x,t)$ still denotes the quantity of polymers of size $x$ at time $t.$

With this model, the problem of PMCA improvement becomes a mathematical optimization problem:
find a control $\af(t)$ which maximizes the quantity $\int xu(T,x)\,dx$ (total mass of pathogenic proteins) at a fixed final time $T.$
The answer to this problem is difficult and a first natural simplification is to consider a control $\af(t)\equiv\af$ which does not depend on time.
In this case the control is a parameter for Equation~\eqref{eq:PMCA} and
the optimization of the payoff $\int xu(T,x)\,dx$ for a large time $T$ reduces to the optimization of the fitness $\Lambda_\af$ of the population. 
Is $\af_{max}$ the best constant to maximize $\Lb_\af?$
Is there a compromise $\af_{opt}\in(1,\af_{max})$ to be found?
The answer depends on the coefficients $\tau$ and $\beta$ as indicated by the different theorems presented in this paper.
More precisely, Theorem~\ref{th:self} ensures that an optimum $\af_{opt}$ can exist between $1$ and $\af_{max}$, and an example is proposed below.

\

{\it\underline{Example.}}
We consider the same coefficients as in Figure~\ref{fig:frag} and suppose that the sonicator can multiply by $4$ the fragmentation at its maximal power.
Then in our model $\af_{max}=4$ and we can see in Figure~\ref{fig:PMCA} that the best strategy to maximize the fitness with a constant coefficient is not the maximal power, 
but an intermediate $\af_{opt}$ between $1$ and $\af_{max}.$

\begin{figure}[ht]
\begin{center}
\begin{minipage}{.9\textwidth}
\psfrag{lambda}[l]{$\Lb_\af$}
\psfrag{alopt}[l]{$\af_{opt}$}
\psfrag{4}[c]{$\af_{max}$}
\psfrag{alpha}[l]{$\af$}
\psfrag{1}[l]{$1$}
\centering\includegraphics [width=.6\linewidth]{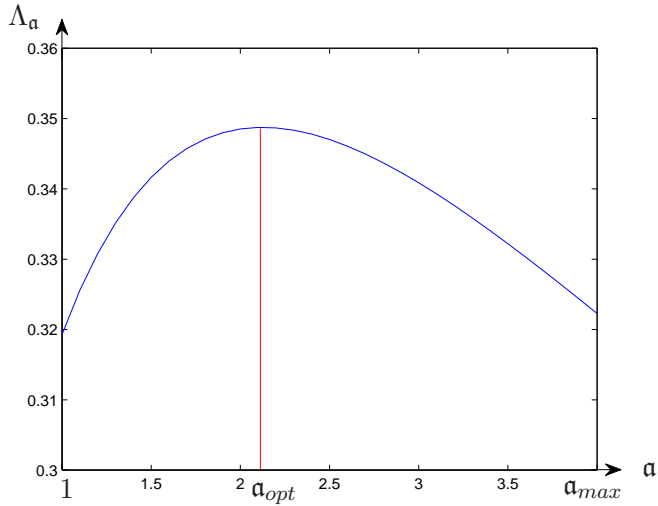}
\caption{The fitness is plotted as a function of $\af$ for the coefficients of Figure~\ref{fig:frag}.
There is a sonication value $\af_{opt}$ in the interior of the window $[1,\af_{max}]$ which maximizes this fitness.
}
\label{fig:PMCA}
\end{minipage}
\end{center}
\end{figure}

This value $\af_{opt}$ should be computed from experimental values of the coefficients, and the corresponding strategy would then consist in a permanent sonication with this optimal power.
But, due to the heat generated by the sonicator, it is not possible to sonicate constantly throughout the entire experiment.
This is why the experimentalists use a periodic protocol with "rest" phases  during which the sample cools down.
The value $\af_{opt}$ may provide informations on the optimal ratio between the time of incubation and sonication phases.
For instance one could try a rectangular control such that the time average is equal to $\af_{opt}.$
More generally, the question is whether we can do as well as or even better than the constant $\af_{opt}$ with a periodic control.
The Floquet theory (see~\cite{BP} for instance) provides a principal eigenvalue for periodic controls.
The comparison between the Floquet and the Perron eigenvalues has to be investigated to optimize the PMCA technique (see~\cite{Lepoutre,Clairambault} for such comparisons on cell cycle models).
The question becomes one of finding a periodic control with an associated Floquet eigenvalue that is as close as possible to $\Lambda_{\af_{opt}},$ or even better.
This problem is adressed in~\cite{CalvezGabriel} for a discrete model, for which the Floquet theory is well established.
The link is made between the eigenvalue optimization problem (for constant and periodic controls) and the optimal control for a final time $T<\infty,$
which is to optimize the total mass $\int xu(T,x)\,dx.$
Different situations are observed where $\af_{opt}$ is the best control or can be improved using a periodic control.

\subsection{Therapeutic optimization for a cell population}\label{ssec:appl:cell}
When Problem \eqref{eq:temporel} models the evolution of a size-structured cell population (or yet a protein-, label-, parasite-structured population), $\tau$ represents the growth rate of the cells and $\beta$ their division rate. It is of great interest to know how a change on these rates can affect the Malthus parameter of the total population, see for instance \cite{Lepoutre,Clairambault}. It is possible to act on the growth rate by changing the nutrient properties - the richer the environment, the faster the growth rate of the cells. We can model such an influence by Equation \eqref{eq:eigen:alpha:tau}, and the question is then how to make $\lb_\al$ as large (if we want to speed up the population growth, for instance for tissue regeneration) or as small (in the case of cancerous cells)  as possible.

Plausible assumptions (see \cite{DMZ} for instance for the case of a size-structured population of E. Coli) for the growth of individual cells is that it is exponential up to a certain threshold, meaning that $\tau(x)=\tau x$ in a neighborhood of zero, and tending to a constant (or possibly vanishing) around infinity, meaning that the cells reach some maximal size or protein-content, leading to $\tau \underset{x\to\infty} \to \tau_\infty <+\infty.$ 

Concerning the division rate $\beta,$ it most generally vanishes around zero, either of the form $\beta(x) \sim \beta x^\gamma$ with $\gamma>1$ or with support $[b,\infty]$ with $b>0,$ 
and it has a maximum, and then decreases for large $x$ - 
and vanishes. Note that for $\tau,$ as for $\beta,$ very little is known regarding their precise behavior for large sizes $x,$ since such values are very rarely reached by cells in the real world.

These assumptions allow us to apply our results. Theorem~\ref{th:coself} and Corollary~\ref{co:limit} lead to vanishing Malthus parameter $\lb_\al$ either for $\al\to 0$ or for $\al\to + \infty.$ This means that 
for cancer drugs, stressing the cells by diminishing nutrients can 
be efficient, 
which is very intuitive and it is known and used for tumor therapy 
(by preventing 
vascularization for instance). What is less intuitive is that forcing 
tumor cells to grow too rapidly in size could also reveal an efficient strategy, as soon as it is established that the division rate decreases for large sizes (this last point could be studied by inverse problem techniques, see \cite{PZ,DPZ,DMZ}). It 
recalls the same ideas as for 
prions, as discussed in Section \ref{ssec:appl:prion}.

In contrast, 
in order to optimize tissue regeneration for instance, these results tend to prove that there exists an optimal value for $\alpha$ such that the Malthus parameter is maximum. This value can be established numerically (see Section \ref{ssec:appl:num} and \cite{GabrielTine}) as soon as the shape of the division rate is known, for instance by using the previously-mentioned inverse problem techniques.

\section*{Conclusion}

The first motivation of our research was to investigate the dependence of the dominant eigenvalue of Problem \eqref{eq:temporel} upon the coefficients $\beta$ and $\tau,$ since a first 
and erroneous idea, based on simple cases, was that it should be monotonic (see \cite{CL1,CL2}). By the use of a self-similar change of variables, we have explored the asymptotic behavior of the first eigenvalue when fragmentation dominates the transport term or \emph{vice versa}. This lead us to counter-examples, 
 where the eigenvalue depends on the coefficients in a non-monotonic way. 
Moreover, these counter-examples are far from being exotic and seem perfectly plausible in many applications, as shown in Section \ref{sec:appl}.  A still open problem is thus to find what would be necessary and sufficient assumptions on $\tau$ and $\beta,$ or better still on the ratio $\f{x\beta}{\tau},$ so that $\lb_\al$ or $\Lb_ \af$ would indeed be  monotonic with respect to $\al$ or $\af.$

Concerning our assumptions, a first glance at the statement of Theorem \ref{th:coself} gives the impression  that only the behavior of the fragmentation rate $\beta$ plays a role in the asymptotic behavior of $\lb_\al,$ and only the ratio $\f{\tau}{x}$ in the behavior of $\Lb_\af.$ This seems puzzling and counter-intuitive. In reality, things are not that simple: to ensure the well-posedness of  eigenvalue problems \eqref{eq:eigen:alpha:tau} and \eqref{eq:eigen:alpha:beta}, Assumptions \eqref{as:betatau0} and \eqref{as:betatauinf} strongly link $\tau$ with $\beta,$ so that a dependence on $\beta$ hides a dependence on $\tau$ and \emph{vice versa.} Moreover, the mathematical techniques used here (moment estimates, multiplication by polynomial weights) force us to restrict ourselves to the space $\cal P$ of functions of polynomial growth or decay. The questions of how 
to relax these (already almost optimal, as shown in \cite{DG}) assumptions and how, if possible, 
to express them in terms of a pure comparison between $\tau,$ $\kappa$ and $\beta$ like in Assumptions \eqref{as:betatau0} and \eqref{as:betatauinf} are still open.

\

{\bf Acknowledgment.} The research of M. Doumic and P. Gabriel is supported by the Agence Nationale de la Recherche, Grant No. ANR-09-BLAN-0218 TOPPAZ. We thank Glenn Webb for his corrections. 

\section*{Appendix 1: technical lemmas}


Here we assume some slight generalizations of Assumption \eqref{as:taunu:betagamma:L}, as 
specified in each lemma, in order to make it 
clear where each part of Assumption \eqref{as:taunu:betagamma:L} is necessary. The other assumptions are those of Theorem \ref{th:self}. 

Lemmas \ref{lm:tech1}, \ref{lm:tech2} and \ref{lm:tech3} prove, respectively,    
that $\beta_\alpha,$ $\tau_\alpha$ and $\kappa_\alpha$ satisfy Assumption \eqref{as:betatauinf}, \eqref{as:betatauspace} and \eqref{as:kappatau}-\eqref{as:betatau0} \emph{uniformly} for all $\alpha.$

\begin{lemma}\label{lm:tech1}
Suppose 
\beq\label{as:taunu:betagamma:L:gen1}
\exists\;\nu,\;\gamma\in\R\quad\text{s.t.}\quad 1+\gamma-\nu >0,\quad\text{and}\quad \frac{\tau(x)}{\beta(x)} \underset{x\rightarrow L}{=} O( x^{\nu-\gamma})\, .
\eeq
Then for all $r>0,$ there exist $A_r>0$ and $\mathcal{N}_L$ a neighborhod of $L$ such that
\beq \text{for a.e.}\ x\geq A_r \ \text{and for all}\ \alpha \in \mathcal{N}_L,\quad  \f{x\beta_\al(x)}{\tau_\al (x)} \geq r.\eeq

\end{lemma}

\begin{proof}
Let $r>0.$ \\
For all $\al>0,$  
using Assumption \eqref{as:betatauinf} we define 
$$p_\al:=\inf\left\{p\,:\, \f{x\beta(x)}{\al \tau(x)} \geq  r ,\ \text{for}\ a.\,e.\ x\geq p\right\}.$$
Observe that $p_\alpha$ is nondecreasing.
Let $\e>0.$ 
By definition of $p_\al,$ there exists a sequence $\{\xi_\al\}$ with values in $[1-\e,1]$ such that
\beq\label{eq:definf}           \f{\xi_\al p_\al\beta(\xi_\al p_\al)}{\al \tau(\xi_\al p_\al)} \leq  r  .\eeq
Assumptions \eqref{as:betatauspace} and \eqref{as:positivity} 
ensure  $\dis\f{\tau}{x\beta}\in L^\infty_{loc}(\R_+^*)$ so that   $\dis p_\al\xrightarrow[\al\to +\infty]{}+\infty$ (
otherwise, since it is nondecreasing, it would tend to a finite limit, 
which contradicts the definition of $p_\al$). We also have
 $\dis\f{\tau}{x\beta} >0$ on $\R_+^*$ so that $\dis p_\al\xrightarrow[\al\to 0]{}0.$ Hence, for some constant $C$ 
\beq\label{eq:equivtaubetap}\f{\tau(\xi_\al p_\al)}{\beta(\xi_\al p_\al)}\underset{\al\to L}{\leq }C (\xi_\al p_\al)^{\nu-\gamma}.\eeq
Then, for some absolute constant $C$, Inequalities \eqref{eq:definf}-\eqref{eq:equivtaubetap} lead to  
$$(1-\e) \f{p_\al}{\al^k}\leq\f{\xi_\al p_\al}{\al^k}\leq \left[C \f{\xi_\al p_\al\beta(\xi_\al p_\al)}{\al \tau(\xi_\al p_\al)}\right]^k\leq  C^k  r^k$$
which implies that $\dis\limsup_{\al\to L}\f{p_\al}{\al^k}\leq C^k  r^k$ is finite. We can define $A_r$ by
$$A_r:=1+\limsup_{\al\to L}\f{p_\al}{\al^k}.$$
Then for any $x>A_r,$  when $\al\to L,$ we have $\al^{k}x>p_\al$ and so
$$\f{\al^k x\beta (\alpha^k x)}{\alpha \tau(\alpha^kx)}\geq r, $$
and by definition of $\beta_\al$ and $\tau_\al$ we obtain the desired result.
\end{proof}

\begin{lemma}\label{lm:tech2}
Suppose that
\beq\label{as:taunu:L:0}
\exists\;\nu\in\R\quad\text{s.t.}\quad \tau(x)\underset{x\rightarrow L}{=} O( x^{\nu})\quad\text{and}\quad\exists\,r_0>0\quad\text{s.t.}\quad x^{r_0}\tau(x)\in L_{loc}^\infty(\R_+).
\eeq
Then for all $A>0$ and $r\geq\max{(r_0,-\nu)},$ there exist $C>0$ and ${\mathcal N}_L$ a neighborhood of $L$ such that
$$\text{for a.e.}\ x\in[0,A]\ \text{and for all}\ \al\in{\mathcal N}_L,\quad x^r\tau_\al(x)\leq C.$$
Suppose that
\beq\label{as:taunu:L:inf}
\exists\;\nu\in\R\quad\text{s.t.}\quad \tau^{-1}(x)\underset{x\rightarrow L}{=} O( x^{-\nu})\quad\text{and}\quad\exists\,\mu>0\quad\text{s.t.}\quad \inf_{x\in[1,+\infty)}x^{\mu}\tau(x)>0.
\eeq
for all $\e>0$ and $m\geq\max(\mu,-\nu),$ there exist $c>0$ and ${\mathcal N}_L$ a neighborhood of $L$ such that
$$\text{for a.e.}\ x\geq\e\ \text{and for all}\ \al\in{\mathcal N}_L,\quad x^{m}\tau_\al(x)\geq c.$$
\end{lemma}

\begin{proof}
We 
treat separately the case $L=0$ and $L=+\infty.$

Let us start with $L=0.$ Notice that in this case, if $\tau(x)= O( x^{\nu}),$ 
 then 
 due to Assumption \eqref{as:betatauspace}, 
  $r_0=-\nu,$ 
  we obtain \eqref{as:taunu:L:0}.
Considering $r\geq-\nu,$ for some constant $C>0$ it holds 
\begin{eqnarray*}
\supess_{x\in[0,A]}\,(x^{r}\tau_\al(x))&=&\supess_{y\in[0,\al^kA]}(\al^{-k(r+\nu)}y^r\tau(y))\\
&\underset{\al\to0}{\leq}&C\sup_{[0,\al^kA]}(\al^{-k(r+\nu)}y^{r+\nu})=CA^{r+\nu}.
\end{eqnarray*}
For $m\geq\max(\mu,-\nu)$ and using Assumption~\eqref{as:taunu:L:inf},  for some constants $c_1,c_2>0,$ we have
\begin{eqnarray*}
\infess_{x\in[\e,\infty)}\,(x^{m}\tau_\al(x))&=&\infess_{y\in[\al^k\e,\infty)}(\al^{-k(m+\nu)}y^{m}\tau(y))\\
&\geq&\min\left(\infess_{[\al^k\e,1]}\,(\al^{-k(m+\nu)}y^{m}\tau(y)),\infess_{[1,\infty)}\,(\al^{-k(m+\nu)}y^{\mu}\tau(y))\right)\\
&\underset{\al\to0}{\geq}&\min\left(c_1\inf_{[\al^k\e,1]}(\al^{-k(m+\nu)}y^{m+\nu}),c_2\al^{-k(m+\nu)}\right)=\min\left(c_1\e^{m+\nu},c_2\al^{-k(m+\nu)}\right)
\end{eqnarray*}

Now we consider $L=+\infty$ and $r\geq\max{(r_0,-\nu)}.$ 
Due to Assumption~\eqref{as:taunu:L:0},  for some constants $C_1,C_2>0,$ we have 
\begin{eqnarray*}
\supess_{x\in[0,A]}\,(x^{r}\tau_\al(x))&=&\supess_{y\in[0,\al^kA]}(\al^{-k(r+\nu)}y^r\tau(y))\\
&\leq&\supess_{[0,1]}\,(\al^{-k(r+\nu)}y^{r_0}\tau(y))+\supess_{[1,\al^kA]}\,(\al^{-k(r+\nu)}y^{r}\tau(y))\\
&\underset{\al\to \infty}{\leq}&C_1\al^{-k(r+\nu)}+C_2A^{r+\nu}.
\end{eqnarray*}
For $m\geq-\nu$ and using Assumption~\eqref{as:taunu:L:inf},  for some $c>0$ we obtain 
\begin{eqnarray*}
\infess_{x\in[\e,\infty)}\,(x^{m}\tau_\al(x))&=&\infess_{y\in[\al^k\e,\infty)}(\al^{-k(m+\nu)}y^{m}\tau(y))\\
&\underset{\al\to\infty}{\geq}&c\inf_{[\al^k\e,\infty)}(\al^{-k(m+\nu)}y^{m+\nu})=c\e^{m+\nu}.
\end{eqnarray*}
\end{proof}

\begin{lemma}\label{lm:tech3}
Suppose that
\beq\label{as:taunu:betagamma:L:gen2}
\exists\;\nu,\;\gamma\in\R\quad\text{s.t.}\quad \gamma+1-\nu >0,\quad\text{and}\quad \frac{\beta(x)}{\tau(x)}\underset{x\rightarrow L}{=} O( x^{\gamma-\nu}).
\eeq
Then for all $\rho>0,$ there exist $\e>0$ and $\mathcal{N}_L$ a neighborhood of $L$ such that
$$\forall\al\in\mathcal{N}_L,\qquad\int_0^\e\f{\beta_\al(x)}{\tau_\al(x)}\,dx\leq\rho.$$
\end{lemma}

\begin{proof}
Due to Assumption~\eqref{as:taunu:betagamma:L:gen2},  for some constant $C>0,$ we have
$$\int_0^\e\f{\beta_\al(x)}{\tau_\al(x)}\,dx=\f1\al\int_0^{\e\al^k}\f{\beta(x)}{\tau(x)}\,dx\underset{\al\to L}{\leq}C\f1\al\int_0^{\e\al^k}x^{\gamma-\nu}dx= C k\e^{\f1k}.$$
The result follows for $\e$ small enough.
\end{proof}

\section*{Appendix 2: Relaxed Case}

In the same spirit as in Lemmas \ref{lm:tech1} to \ref{lm:tech3}, we relax Assumption \eqref{as:taunu:betagamma:L} and 
examine if the asymptotic behavior of $\lb_\al$ and $\Lambda_\af$ obtained in Theorem~\ref{th:coself} remains 
 true.
The case we are the most interested in is the case when the limits are zero (see the applications at Section~\ref{sec:appl}). 
Is the condition $\dis\lim_{x\to L}\beta(x)=0\ (\text{resp.}\lim_{x\to\f1L}\f{\tau(x)}{x}=0)$ necessary and sufficient to have $\dis\lim_{\al\to L}\lambda_\al=0\ (\text{resp.}\ \lim_{\af\to L}\Lb_\af=0)?$
The following proposition gives partial results in the direction of a positive answer to this question.
The assumptions required are weaker, but the results also are 
weaker. 
We obtain asymptotic behavior for the eigenvalue, 
 but cannot say anything yet on the eigenvector behavior.

\begin{proposition}
 \label{th:relax}
Let us suppose that all assumptions of Theorem \ref{th:self} are satisfied except Assumption \eqref{as:taunu:betagamma:L}.
\begin{enumerate}
\item\label{th:relax:xzero}
If $\tau(x) \underset{x\rightarrow0}{=} o( x^\nu),\ \beta\underset{x\rightarrow0}{=} O(x^\gamma)$ and $\beta(x)^{-1}\underset{x\rightarrow0}{=} O( x^{-\gamma})$ with $\gamma+1-\nu>0,$ we have 
$$\text{if}\ \gamma>0,\ \text{then}\ \lim_{\al\to0}\lb_\al=0, \ \text{and more precisely} \quad \lb_\al \underset{\al\rightarrow 0}{=} o(\al^{\f{\gamma}{1+\gamma-\nu}}),$$

$$\text{if}\ \nu\geq1,\ \text{then}\ \lim_{\af\to\infty}\Lb_\af=0,\ \text{and more precisely}\quad \Lb_\af\underset{\af\rightarrow \infty}{=}o(\af^{\f{1-\nu}{1+\gamma-\nu}}).$$

\item\label{th:relax:xinfinity}
If $\beta(x)\underset{x\rightarrow\infty}{=} o( x^\gamma)$ and $\tau(x)^{-1}\underset{x\rightarrow\infty}{=} O( x^{-\nu})$ with $\gamma+1-\nu>0$ and $\gamma\leq0$ (so that $\nu<1$), we have 
$$\lim_{\al\to\infty}\lb_\al=\lim_{\af\to0}\Lb_\af=0,\ \text{and more precisely} \quad \lb_\al \underset{\al\rightarrow +\infty}{=} o(\al^{\f{\gamma}{1+\gamma-\nu}}),\ \Lb_\af \underset{\af\rightarrow0}{=} o(\af^{\f{1-\nu}{1+\gamma-\nu}}).$$
\end{enumerate}

\end{proposition}

{\bf Remark.} In the second assertion of this proposition, we notice that Assumption~\eqref{as:betatauinf} means $\f{\tau(x)}{x}\underset{x\rightarrow\infty}{=} o(\beta(x)),$ so the condition $\beta(x)=o( x^\gamma)$ with $\gamma\leq0$ imposes $\lim_{x\to\infty}\f{\tau(x)}{x}=0.$

\begin{proof}[Proof of Proposition \ref{th:relax}.\ref{th:relax:xzero}.]
We perform the dilation defined by \eqref{eq:defval}: $v_\al(x)=\al^k\U_\al(\al^k x)$ with $k=\f{1}{1+\gamma-\nu}.$
Due to the assumption $\beta(x)^{-1}\underset{x\rightarrow0}=O(x^{-\gamma})$ and $\tau=O(x^\nu),$ the conclusions of Lemma~\ref{lm:tech1} and Lemma~\ref{lm:tech2} still hold true.
Hence we have the following bound (see the first estimate in the proof of Theorem \ref{th:self})
$$\int x^r\beta_\al(x)v_\al(x)\,dx\leq\f{r \sup_{(0,A_\f{r}{\omega})}(x^{r-1}\tau_\al)}{1-2c-\omega}$$
where $\tau_\al(x)$ and $\beta_\al$ are defined by \eqref{eq:def:taual}, and the right-hand 
side is bounded uniformly in $\al$ for $r\geq\max{(2,1+r_0,1-\nu)}.$
Let $\e>0$ and write
\begin{eqnarray*}
 \alpha^{-\f{\gamma}{1+\gamma-\nu}} \lb_\al=\theta_\al=\int\beta_\al v_\al&\leq&\int_0^\e\beta_\al(x)v_\al(x)\,dx+\e^{-r}\int_0^\infty x^r\beta_\al(x)v_\al(x)\,dx\\
&\leq&\sup_{(0,\e)}\beta_\al+\e^{-r}\f{r\sup_{(0,A_{\f r\omega})}(x^{r-1}\tau_\al)}{1-2c-\omega}.
\end{eqnarray*}
Thus, since $\dis\sup_{(0,A_{\f{r}{\omega}})} x^{r_0} \tau_\al\underset{\al\to0}{\longrightarrow}0,$ we obtain
$$\limsup_{\al\to0}\theta_\al\leq\limsup_{\al\to0} \sup_{(0,\e)}\beta_\al\leq C \e^\gamma.$$
This is true for all $\e>0,$ 
so  Assertion~\ref{th:relax:xzero} of Proposition~\ref{th:relax} is proved (the same proof works with the fragmentation parameter $\af$).

\end{proof}

\begin{proof}[Proof of Proposition \ref{th:relax}.\ref{th:relax:xinfinity}.]

We perform the dilation defined by \eqref{eq:defval} $v_\al(x)=\al^k \U_\al(\al^k x)$ with $k=\f{1}{1+\gamma-\nu}.$ 
Due to Assumption \eqref{as:taunu:betagamma:L:gen2} for $L=+\infty,$ we still have the conclusion of Lemma~\ref{lm:tech3} and it is sufficient to bound $\tau_\al v_\al$ on $(0,\e)$ for $\e>0.$ We refer to the proof of Theorem \ref{th:self} in Section \ref{ssec:proof:self}, second estimate, and write
\begin{eqnarray*}
\tau_\al(x)v_\al(x)&\leq&2\sup_{(0,\e)}\{\tau_\al v_\al\}\int_0^\e\f{\beta_\al(y)}{\tau_\al(y)}\,dy+ 2\int_\e^\infty\beta_\al(y)v_\al(y)\,dy\\
&\leq&2\rho\sup_{(0,\e)}\{\tau_\al v_\al\}+2 \sup_{(\e,+\infty)} \beta_\al .
\end{eqnarray*}
Taking for instance $\e$ small enough so that $\rho\leq\f{1}{4}$ in this estimate, and $\al$ large enough so that $\sup_{(\e,+\infty)} \beta_\al \leq C,$ we obtain the boundedness of $\tau_\al v_\al$ on $(0,\e_0).$ 
Then for $\e>0$ we write
\begin{eqnarray*}
\al^{-\f{\gamma}{1+\gamma-\nu}}\lb_\al=\theta_\al&=&\int\beta_\al v_\al\\
&\leq&\underbrace{\sup_{(0,\e)}\tau_\al v_\al}_{\leq C}\underbrace{\int_0^\e\f{\beta_\al}{\tau_\al}}_{\xrightarrow[\e\to0]{}0}+\underbrace{\sup_{(\e,\infty)}\beta_\al}_{\xrightarrow[\al\to\infty]{}0}.
\end{eqnarray*}
The latter estimate is a consequence of the assumption $\beta(x)\underset{x\rightarrow\infty}{=} o( x^\gamma)$.

We do the same computations for the parameter $\af.$

\end{proof}

%
%
%

\

%
%

\bibliographystyle{abbrv}
\bibliography{Prion}

\end{document}